\documentclass[proquest]{uwthesis}[2014/11/13]

\usepackage{amssymb,amsmath,amsthm}

\usepackage{microtype}
\usepackage{comment}
\usepackage{enumerate}
\usepackage{multirow}

\DeclareMathOperator\im{Im}
\renewcommand{\hom} {\operatorname{Hom}}
\DeclareMathOperator\coker{coker}
\DeclareMathOperator\sym{Sym}
\DeclareMathOperator\aut{Aut}
\DeclareMathOperator\gr{gr}

\DeclareMathOperator\chr{char}
\DeclareMathOperator\len{len}

\DeclareMathOperator\Span{span}
\DeclareMathOperator\hgt{ht}
\DeclareMathOperator\id{id}
\DeclareMathOperator\gal{Gal}
\DeclareMathOperator\fr{Fr}
\DeclareMathOperator\Sq{Sq}
\DeclareMathOperator\N{\mathcal{N}}

\newcommand{\iso} {\approx}

\newcommand{\normal} {\triangleleft}
\newcommand{\semidirect} {\rtimes}
\newcommand{\tr}[2] {\operatorname{tr}\!\big\uparrow_{#1}^{#2}}

\newcommand{\res}[2] {\operatorname{res}\!\big\downarrow_{#1}^{#2}}
\newcommand{\inv} {\ensuremath ^{-1}}
\newcommand{\bbR} {\ensuremath{\mathbb{R}}}
\newcommand{\bbC} {\ensuremath{\mathbb{C}}}
\newcommand{\bbQ} {\ensuremath{\mathbb{Q}}}
\newcommand{\bbZ} {\ensuremath{\mathbb{Z}}}
\newcommand{\bbN} {\ensuremath{\mathbb{N}}}
\newcommand{\F} {\ensuremath{\mathbb{F}}}
\newcommand{\Ga} {\ensuremath{\mathbb{G}_a}}
\newcommand{\Fbar} {\ensuremath{\overline{\mathbb{F}}}}

\newcommand{\powerset} {\ensuremath{\mathcal{P}}}
\newcommand{\set}[2] {\left\lbrace {#1} \,\middle\arrowvert\, {#2} \right\rbrace}

\newcommand{\mat}[2] {\begin{bmatrix} #2 \end{bmatrix}}

\newcommand{\+} {\quad\mbox{and}\quad}

\newcommand\numberthis{\addtocounter{equation}{1}\tag{\theequation}}

\theoremstyle{definition}
\newtheorem{thm}{Theorem}
\newtheorem{prop}[thm]{Proposition}
\newtheorem{lemma}[thm]{Lemma}

\newtheorem{cor}[thm]{Corollary}
\newtheorem{defn}[thm]{Definition}
\newtheorem{rem}[thm]{Remark}
\newtheorem{quest}[thm]{Question}

\newtheoremstyle{named}{}{}{}{}{\bfseries}{.}{.5em}{#3}
\theoremstyle{named}
\newtheorem*{namedthm}{}

\numberwithin{thm}{chapter}

\begin{document}

\prelimpages



\Title{Some cohomology of finite general linear groups}
\Author{David Sprehn}
\Year{2015}
\Program{Mathematics}

\Chair{Steve Mitchell}{Professor}{Mathematics}
\Signature{Julia Pevtsova}
\Signature{John Palmieri}

\titlepage
\newpage
\copyrightpage

\setcounter{page}{-1}
\abstract{
We prove that the degree $r(2p-3)$ cohomology of any (untwisted) finite group of Lie type over $\F_{p^r}$, with coefficients in characteristic $p$, is nonzero as long as its Coxeter number is at most $p$.  We do this by providing a simple explicit construction of a nonzero element.
Furthermore, for the groups $GL_n\F_{p^r}$, when $p=2$ or $r=1$, we calculate the cohomology in degree $r(2p-3)$, for all $n$.
}

\addtocontents{toc}{\protect\enlargethispage{-20pt}}
\tableofcontents


\acknowledgments{
  I would like to thank Steve Mitchell, my thesis advisor, for introducing me to the wonderful subject of group cohomology, for encouraging me to think about finite groups of Lie type, and for many absolutely wonderful discussions.
  I also thank Milnor and Stasheff for their book on characteristic classes which, together with Steve Mitchell's teaching, shaped my mathematical taste so far.
}

%
%
\dedication{\begin{center}to my parents, Susan Morris and Randy Simpson\end{center}}

\textpages

\section{Introduction}
The general linear groups $GL_n(\F_q)$ are among the most important finite groups.
They are easy to define (the invertible matrices over the finite field of order $q=p^r$), but their ($p$-)subgroup structure is very rich and mysterious, and not amenable to direct study.  For instance, conjugacy classes of subgroups isomorphic to $E=C_p^k$ correspond to isomorphism classes of (faithful) modular representations of $E$, and classifying those (when $k>2$) has been proven algorithmically impossible~\cite[sec.~4.4]{Benson}.  Group cohomology provides a very powerful, coarser invariant which often succeeds in giving useful information when direct approaches fail.  It is an invaluable tool in many fields, including algebraic topology (where the homotopy types of classifying spaces are of interest), representation theory (via characteristic classes and support varieties), group theory, and number theory.
However, there is a big hole in our knowledge of the cohomology of general linear groups.

The cohomology at primes other than $p$,
\[ H^*(GL_n\F_{p^r};\F_\ell),\quad \ell\neq p, \]
was completely determined by Quillen~\cite{QuillenK} in the context of his work on the algebraic K-theory of finite fields: the answer is fairly simple, and not too far in spirit from the classical ``splitting principle'' calculation over $\bbR$ or $\bbC$.  In the simplest case, when $\ell\mid(p^r-1)$, the mod-$\ell$ cohomology is detected on the subgroup of diagonal matrices in $GL_n\F_{p^r}$; in all cases it is detected on a single abelian subgroup.\footnote{Namely, the subgroup of ($\ell$-torsion) diagonal matrices in $GL_m\F_{p^{rk}}\leq GL_{mk}\F_{p^r}\leq GL_n\F_{p^r}$, where $\F_{p^{rk}}$ is the field extension of $\F_{p^r}$ obtained by adjoining the $\ell$th roots of unity, and $m=\lfloor n/k\rfloor$.}

In contrast, Quillen did not compute the mod-$p$ cohomology
\[ H^*(GL_n\F_{p^r};\F_p), \]
calling it a ``difficult problem.''
This mod-$p$ cohomology has continued to resist computation and, four decades later, we still know very little about it.
There are very few instances of known nonzero cohomology classes.  Milgram and Priddy~\cite{MP} constructed a family of algebraically independent classes (when $r=1$), in very high degree.  Barbu~\cite{Barbu} constructed a class (again when $r=1$) in degree $2p-2$, nonzero for $2\leq n\leq p$.  Bendel, Nakano, and Pillen~\cite{BNP,BNP2} proved (nonconstructively) that there is a nonzero class in degree $r(2p-3)$ when $2\leq n<p-1$, and more generally studied the problem of where the first nonzero cohomology occurs on finite groups of Lie type, under the assumption that the Coxeter number is small compared to $p$.  Finite groups of Lie type can be thought of as finite-field analogs of Lie groups, and include e.g.~the general linear groups, special linear groups, symplectic groups, and orthogonal groups over the finite fields.  Coxeter number is a measure of size, in terms of the root system, and therefore a measure of the complexity (specifically, nilpotency order) of the $p$-Sylow subgroup.  The Coxeter number of $GL_n\F_p$ is $n$.

In this dissertation, we investigate a number of issues related to the mod-$p$ cohomology of $GL_n\F_{p^r}$ in degree $r(2p-3)$, the lowest degree in which it may possibly be nonzero (for any $n$), as well as the mod-$p$ cohomology in this degree of other finite groups of Lie type over $\F_{p^r}$.

\subsection{Vanishing ranges}
There are two types of inclusions among the groups $GL_n{\F_{p^r}}$, those increasing $n$ and those increasing $r$ (via a field extension).  A striking feature of the situation is that the cohomology rings vanish in both limits.  That is,
\begin{align*}
H^i\left(\lim_n GL_n\F_{p^r};\F_p\right) &=
H^i\left(GL\F_{p^r};\F_p\right)=0 \+ \\
H^i\left(\lim_r GL_n\F_{p^r};\F_p\right) &=
H^i\left(GL_n\overline{\F_p};\F_p\right)=0
\end{align*}
for all $i>0$.  Both of these results are due to Quillen~\cite{QuillenK}.  
This stable vanishing means that the approaches used to compute the cohomology rings of the analogous groups in characteristic other than $p$ (and, classically, linear groups over $\bbR$ and $\bbC$) can't be applied, as there is no base of stable classes from which to begin.

Nevertheless, the cohomology rings $H^i(GL_n\F_{p^r};\F_p)$ of the finite groups are very far from trivial, as they have positive Krull dimension: $r\lfloor n^2/4\rfloor$.  We know this because the Atiyah-Swan conjecture, proven by Quillen~\cite{QuillenS}, says that the Krull dimension of the mod-$p$ group cohomology ring of a finite group $G$ (or, more generally, compact Lie group) equals the maximal $k$ such that $G$ contains a subgroup isomorphic to $C_p^k$.  For $G=GL_{2m}\F_{p^r}$,
\[ \begin{bmatrix} I_m&* \\ 0&I_m \end{bmatrix} \]
is such a subgroup.  This maximal rank of an elementary abelian $p$-subgroup in $GL_n\F_{p^r}$ was computed by Milgram and Priddy~\cite{MP} when $r=1$, although their proof is equally valid for $r>1$.  It follows that the cohomology ring contains a set of $r\lfloor n^2/4\rfloor$ algebraically independent cohomology classes.  Milgram and Priddy (when $r=1$) constructed such a family; however, they are in exponentially high degree (at lowest, $2p^{n-2}(p-1)$ for $p$ odd).

The two ``vanishing in the limit'' statements above do strongly suggest that each cohomology group $H^i(GL_n\F_{p^r};\F_p)$ for fixed $i$ is zero for large enough $r$, or for large enough $n$.  This is true, and there have been a number of such refinements, which give a range in which the cohomology vanishes, growing in terms of either $n$ or $r$.
\pagebreak[2]
The current best vanishing ranges state that
$H^i(GL_n{\F_{p^r}};\F_p)=0$ if:
\nopagebreak
\begin{enumerate}[(V1)]
\singlespace
\item $0<i<r(2p-3)$, or \hfill Friedlander-Parshall~\cite{FP}
\item $0<i<n$ and $p^r\neq 2$, or \hfill Quillen [unpublished]
\item $0<i<n/2$. \hfill Maazen~\cite{Maazen}
\end{enumerate}

\noindent Regarding $p^r$ as constant and varying $n$, it is clear that only range (V1) is relevant when $n$ is small, and only (V2,V3) are relevant when $n$ is large.

In this dissertation, we will investigate several issues related to the vanishing range (V1).  We are not yet able to contribute any new information about the ``large $n$'' situation.

\subsection{Outline}
In chapter~\ref{gl2ch}, we provide a detailed discussion of the (known) cohomology groups of $GL_2\F_{p^r}$, pointing out some features which will be useful later, including a description of the nilpotent elements.

In chapter~\ref{lietypech}, we show that
\[ H^{r(2p-3)}(GL_n\F_{p^r}; \F_p) \neq 0\quad\mbox{for}\quad 2\leq n\leq p. \]
In fact, we provide a simple explicit construction, valid more generally in all (untwisted) finite groups $G$ of Lie type over $\F_{p^r}$ with Coxeter number at most $p$, of a nonzero class in that degree.  The construction is a transfer up from the Borel subgroup of a class pulled back from a quotient.  It shows moreover that $H^*(G;\F_p)$ contains a copy of $H^*(GL_2\F_{p^r};\F_p)$ as a graded vector space, and a module over the Steenrod algebra.

In chapter~\ref{dimensionch}, we focus on general linear groups $GL_n\F_{p^r}$, attempting to completely compute their cohomology in degree $r(2p-3)$, for all $n$.  We succeed in the case $p=2$ or $r=1$, showing that this cohomology is one-dimensional for $2\leq n\leq p$ and zero for $n>p$.  In other words, the results of chapter~\ref{lietypech} are sharp in the sense that they describe all of the nonzero cohomology in degree $r(2p-3)$, for all $n$.  To handle the remaining cases, one would need an understanding of the cohomology of the ``hook subgroup:'' the subgroup of matrices which differ from the identity matrix only in the first row and last column.
Along the way, we digress (section~\ref{char2lietype}) to give a generalization of the Friedlander-Parshall vanishing bound (V1 above) to other Lie types, in characteristic $p=2$.

The strategy of chapter~\ref{dimensionch} is to compare the cohomology of the Sylow $p$-subgroup $U_n$ to the cohomology of its associated graded group for the lower central series, which is elementary abelian.
Both Quillen and Friedlander-Parshall employed this idea to establish vanishing ranges,
by showing that $H^*(\gr U_n;\F_p)$ already has no diagonal torus invariants in the desired range.  Computing these invariants is reduced to a question of combinatorics, because $H^*(\gr U_n;\F_p)$ has a monomial basis of eigenvectors for the diagonal torus.  Our strategy is to extend the machinery to carry out the comparison, not only in the range where these invariants are zero, but also in the lowest degree where they are nonzero.

\subsection{Boundaries of our Lie type definition}
There are many variations on the definition of a finite group of Lie type.  Our definition: the group of fixed-points of a standard Frobenius endomorphism on a connected reductive algebraic group (over $\Fbar_p$).  One benefit of merely requiring the algebraic group to be reductive, as opposed to semisimple or simple, is that the most important examples, the general linear groups, are included with no hassle.  There is no significant extra difficulty in allowing reductive groups.  However, this definition does still exclude the finite simple groups, such as $PSL_n\F_{p^r}$.  This is no big loss, as far as cohomology is concerned: e.g.~$SL_n\F_{p^r}\to PSL_n\F_{p^r}$ induces an isomorphism on $\F_p$-cohomology.  On the other hand, our definition is restrictive in that it excludes the twisted groups (defined with a Steinberg endomorphism rather than a standard Frobenius endomorphism), such as the unitary groups over $\F_{p^r}$.  It is not easy to extend the results of chapter~\ref{lietypech} to these groups, as (for example) the root subgroups of the finite unitary groups are not always abelian!

\chapter{Cohomology of $GL_2$}\label{gl2ch}
\section{Assorted facts}
For reference, we briefly recite some facts about group cohomology and the transfer homomorphism, which will be used frequently throughout this dissertation.

Let $G$ be a finite group.  For any commutative ring $R$, the group cohomology $H^*(G;R)$ is a connected graded $R$-algebra, depending functorially on both $R$ and $G$.  In particular, for $H\leq G$ a subgroup, there is a ``restriction homomorphism''
\[ \res{H}{G}:H^*(G;R)\to H^*(H;R), \]
which is a graded $R$-algebra homomorphism.
Because inner automorphisms of $G$ act trivially on cohomology, we have
\[ \im(\res{H}{G})\leq H^*(G;R)^{N_G(H)/H}, \]
i.e.~the image of restriction to $H$ lies in the invariants by the normalizer of $H$.
There is also a ``transfer map'' going the other way,
\[ \tr{H}{G}:H^*(H;R)\to H^*(G;R), \]
which is a graded $R$-module homomorphism, but not multiplicative.
They are related by two composition formulae:
\begin{enumerate}
\item $\tr{H}{G}\circ\res{H}{G}$ equals multiplication by $[G:H]$ on $H^*(G;R)$;
\item \textit{(Double coset formula)}
\[ \res{K}{G}\circ\tr{H}{G}=\sum_{g\in K\backslash G/H}
\tr{K\cap gHg\inv}{K}\circ\res{K\cap gHg\inv}{gHg\inv}\circ c_{g\inv}^*, \]
where $K\leq G$ is also a subgroup, the sum is over double-coset representatives, and $c_{g\inv}:gHg\inv\to H$ is conjugation.
\end{enumerate}

\noindent
Some consequences of the two formulae:
\begin{enumerate}
\item $H^i(G;\bbQ)=0$ for all $i>0$;
\item If $p$ does not divide $G$, then $H^i(G;\F_p)=0$ for all $i>0$;
\item If $H$ is (or contains) a $p$-Sylow subgroup of $G$, then
\[ \res{H}{G}:H^*(G;\F_p)\to H^*(H;\F_p) \]
is injective;
\item If $H$ contains a $p$-Sylow subgroup of $G$ and is normal in $G$, then
\[ \res{H}{G}:H^*(G;\F_p)\to H^*(H;\F_p)^{G/H} \]
is an isomorphism;\label{normalsylow}
\item If $G$ is an elementary abelian $p$-group and $H$ is a proper subgroup, then
\[ \tr{H}{G}:H^*(H;\F_p)\to H^*(G;\F_p) \]
is zero.
\end{enumerate}

When $H\leq G$ is a normal subgroup, there is a (Lyndon-Hochschild-Serre) spectral sequence
\[ H^p(G/H;H^q(H;R))\Rightarrow H^{p+q}(G;R). \]
If $R=\F_p$ and $p$ does not divide the order of $G/H$, it collapses to the vertical edge, giving another argument for consequence~\ref{normalsylow} above.  If instead $p$ does not divide the order of $H$,  it collapses to the horizontal edge, giving an isomophism
\[ H^*(G/H;\F_p)\xrightarrow{\sim} H^*(G;\F_p). \]
For example, the projections $GL_n\F_{p^r}\to PGL_n\F_{p^r}$ and $SL_n\F_{p^r}\to PSL_n\F_{p^r}$ induce isomorphisms on $\F_p$-cohomology.

The following two facts can be interpreted as special cases of the K\"unneth Theorem (or checked more directly).
\begin{enumerate}
\item\textit{(Universal Coefficient Theorem)} For $F\leq E$ a field extension, $H^*(G;F)\otimes E\to H^*(G;E)$ is an isomorphism of $E$-algebras.
\item Field extension commutes with invariants: for $F\leq E$ a field extension, and $T$ a group acting by automorphisms on $G$,
\[H^*(G;F)^T\otimes E\to H^*(G;E)^T \]
is an isomorphism.
\end{enumerate}

\section{Cohomology of $GL_2\F_{p^r}$}\label{gl2}
We will set the stage by calculating the mod $p$ cohomology of the finite groups $GL_2\F_{p^r}$, and show that $r(2p-3)$ is the first degree in which the cohomology is nonzero.  This will explain the significance of the number $r(2p-3)$ in terms of the invariant theory of finite fields.  The calculation sketched in this section is well known, and has appeared in~\cite{QuillenK,Aguade,FP}.

The unipotent upper triangular subgroup,
\[ U=\begin{bmatrix}1&*\\0&1\end{bmatrix}\iso\F_{p^r}, \]
is a Sylow $p$-subgroup of $GL_2\F_{p^r}$.  It is elementary abelian, which is what enables a comparatively easy calculation in this case.  Its normalizer is the subgroup of upper-triangular matrices,
\[ B_2\F_{p^r}=\begin{bmatrix}*&*\\0&*\end{bmatrix}\iso U\semidirect T, \]
a semidirect product of $U$ with the subgroup $T=(\F_{p^r}^\times)^2$ of diagonal matrices.
We claim that the restriction map
\[ H^*(GL_2\F_{p^r};\F_p)\to H^*(U;\F_p)^T\iso H^*(\F_{p^r};\F_p)^{\F_p^\times} \]
is an isomorphism.
(In the latter, we need only one factor of $\F_p^\times$ because the scalar matrices act trivially on $U$.)
To see this, consider the composition of transfer up from $H^*(U;\F_p)$ to $H^*(GL_2\F_{p^r};\F_p)$ followed by restriction back to $U$.  The double coset formula expresses this as a sum of transfers up from subgroups of $U$.  But, as $U$ is elementary abelian, the transfer up from any proper subgroup is trivial.  Hence the only nontrivial terms are those corresponding to elements of the normalizer of $U$.  So the composition is equal to $\sum_{t\in T} c_t^*$.  As $p$ does not divide the order of $T$, the image of this endomorphism is precisely the $T$-invariants.

All of the same remarks apply with $B_2\F_{p^r}$ in place of $GL_2\F_{p^r}$ as well.  In summary:
\begin{prop} For any prime $p$ and $r\geq1$, we have isomorphisms
\[ H^*(GL_2\F_{p^r};\F_p)\to H^*(B_2\F_{p^r};\F_p)\to H^*(U;\F_p)^T\iso H^*(\F_{p^r};\F_p)^{\F_{p^r}^\times} \]
given by the restriction maps.
\end{prop}

Having expressed the cohomology of $GL_2\F_{p^r}$ as the invariants of the multiplicative group of $\F_{p^r}$ acting on the cohomology of the additive group, we proceed to describe those invariants.

Since $\F_{p^r}\iso C_p^r$ as abelian groups, the K\"unneth Theorem expresses its cohomology as a tensor product of $r$ copies of $H^*(C_p;\F_p)$, which has either one polynomial generator, or one polynomial and one exterior generator.
However, there is no canonical choice of the isomorphism $\F_{p^r}\iso C_p^r$ and I am not aware of any way of expressing the action of the multiplicative group $\F_{p^r}^\times$ on the generators thus obtained.  To get around this issue, we extend coefficients to $\F_{p^r}$ as well.  By the Universal Coefficient Theorem, this does not affect the dimensions of the cohomology groups.  Since $\F_{p^r}$ has all $(p^r-1)$th roots of unity (and they are all distinct), the action of $T$ must diagonalize after extending coefficients.  We proceed by finding a convenient eigenbasis.

First, recall that for a Galois field extension $E$ over $F$, and any $E$-vector space $V$, we have  a natural isomorphism
\[ E\otimes_F V\iso\bigoplus_{g\in\gal(E:F)}g^*(V). \]
Here the tensor product is made a vector space over $E$ via the left factor, and $g^*(V)$ is an $E$-vector space via the action $e\cdot v=g(e)v$.  The naturality is with respect to $E$-vector space homomorphisms $V\to V'$.
(The isomorphism sends $e\otimes v\mapsto(g(e)v)_g$, i.e.~the sum over all $g\in\gal(E:F)$ of $g(e)v$ in the summand $g^*(V)$.)

In our situation, we take $F=\F_p$ and $E=\F_{p^r}$.  Here the Galois group is cyclic of order $r$, generated by the Frobenius map $\fr$.  Hence
\[ \F_{p^r}\otimes_{\F_p}V\iso\bigoplus_{i=0}^{r-1}(\fr^i)^*(V) \]
naturally as vector spaces over $\F_{p^r}$.

Recall that when $A$ is an elementary abelian $p$-group, we have a natural isomorphism
\[ H^*(A;\F_p)=\begin{cases}
\sym^*(A^*) &\mbox{if  }p=2,\\
\Lambda^*(A^*)\otimes \sym^*(\beta A^*) &\mbox{if  }p\mbox{ odd}.
\end{cases} \]
where $A^*=\hom(A,\F_p)=H^1(A;\F_p)$.
The universal coefficient theorem then gives a natural description
\[ H^*(A;\F_{p^r})\iso\begin{cases}
\sym^*(W) &\mbox{if  }p=2,\\
\Lambda^*(W)\otimes \sym^*(W) &\mbox{if  }p\mbox{ odd}.
\end{cases} \]
where $W=\F_{p^r}\otimes_{\F_p}A^*$.
When $p=2$ the symmetric algebra is generated in degree 1; when $p$ is odd, the exterior algebra is generated in degree 1 and the symmetric algebra in degree 2.
Here the exterior and symmetric powers, and the tensor product, are now over $\F_{p^r}$.

Now regard an $\F_{p^r}$-vector space $V$ as an additive group: an elementary abelian $p$-group.  The above two observations combine to give:
\begin{prop}\label{vectorspacecoh}
Let $V$ be a vector space\footnote{Tensor products, duals, and symmetric and exterior algebras are taken as vector spaces over $\F_{p^r}$.}
over $\F_{p^r}$.
\[ H^*(V;\F_{p^r})\iso\begin{cases}
\sym^*\left(\bigoplus_{i=0}^{r-1}(\fr^i)^*(V^*)\right) &\mbox{if $p=2$},\\
\Lambda^*\left(\bigoplus_{i=0}^{r-1}(\fr^i)^*(V^*)\right)\otimes \sym^*\left(\bigoplus_{i=0}^{r-1}(\fr^i)^*(V^*)\right) &\mbox{if $p$ odd}.
\end{cases} \]
When $p=2$ the symmetric generators are in degree 1; when $p$ is odd, the exterior generators are in degree 1 and the symmetric generators in degree 2.
The isomorphism is natural with respect to $F_{p^r}$-linear maps $V\to V'$.
\end{prop}

Being natural in $V$, this description is in particular equivariant with respect to the action of the multiplicative group $\F_{p^r}^\times$ on $V$ by scalar multiplication.

In coordinates, Proposition~\ref{vectorspacecoh} says that:
\[ H^*(\F_{p^r}^n;\F_{p^r}) = \begin{cases}
\F_{2^r}[x_{ik}\mid 1\leq i\leq n,\ 0\leq k<r] &\mbox{if $p=2$},\\
\F_{p^r}\langle x_{ik}\mid 1\leq i\leq n,\ 0\leq k<r\rangle \\
\qquad\otimes
\F_{p^r}[y_{ik}\mid 1\leq i\leq n,\ 0\leq k<r] &\mbox{if $p$ odd},
\end{cases} \]

where $\deg(x_{ik})=1$, $\deg(y_{ik})=2$, and $\lambda\in\F_{p^r}^\times$
acts on $x_{ik}$ or $y_{ik}$ as multiplication by $\lambda^{p^k}$.
This algebra therefore has a basis of monomials in the $x_{ik}$'s (and $y_{ik}$'s), which are all eigenvectors for the action of the multiplicative group.  Consequently the invariants of this action are spanned by the invariant monomials.  The latter cannot occur below degree $r(p-1)$, which we show using the following lemma of Quillen~\cite[Lem.~16]{QuillenK}:
\begin{lemma}\label{Quillenlemma}
Suppose $a_1,\dots,a_k\in\bbN$ are not all zero, and
\[ (p^r-1)\mid\sum_{k=0}^{r-1} p^ka_k. \]
Then \[ \sum_{k=0}^{r-1} a_k\geq r(p-1). \]
In case of equality, $a_0=\cdots=a_{r-1}=p-1$.
\end{lemma}

Now, let\footnote{Here, and in the proof of Proposition~\ref{gl2vanishing}, we write as though $p$ were odd.  If $p=2$, the reader may supply a similar, though simpler, proof.  Alternatively, we may take $y_{ik}=x_{i,k-1}^2$ in this proof.}
\[ z=\prod_{i=1}^n\prod_{k=0}^{r-1}x_{ik}^{a_{ik}}y_{ik}^{b_{ik}}\in H^*(\F_{p^r}^n;\F_{p^r}) \]
be a nontrivial monomial, where $a_{ik}\in\{0,1\}$ and $b_{ik}\in\bbN$.
If $z$ is invariant under the action of $\F_{p^r}^\times\iso C_{p^r-1}$, we see
(by choosing a generator) that
\[ (p^r-1)\mid\sum_{k=0}^{r-1}p^k\left(\sum_i a_{ik}+b_{ik}\right). \]
Hence, using Lemma~\ref{Quillenlemma},
\begin{align*}
\deg(z)&=\sum_{k=0}^{r-1}\left(\sum_i a_{ik}+2b_{ik}\right) \\
&\geq \sum_{k=0}^{r-1}\left(\sum_i a_{ik}+b_{ik}\right) \\
&\geq r(p-1).
\end{align*}
Therefore,
$H^i(\F_{p^r}^n;\F_{p^r})^{\F_{p^r}^\times}=0$ when $0<i<r(p-1)$.

Now, since the ``right edge subgroup''
\[ \F_{p^r}^n\semidirect\F_{p^r}^\times =
\begin{bmatrix} I_n&*\\0&\F_{p^r}^\times\end{bmatrix}\leq GL_{n+1}\F_{p^r} \]
has a normal Sylow $p$-subgroup $\F_{p^r}^n$ with complement $\F_{p^r}^\times$ acting by scalar multiplication, this shows:

\begin{lemma}\label{edgegroup}
$H^i(\F_{p^r}^n\semidirect\F_{p^r}^\times;\F_p)=0$
for $0<i<r(p-1)$ and all $n$.
\end{lemma}

Now we specialize to the case $n=1$, which was what we needed to understand the cohomology of $GL_2\F_{p^r}$.
When $n=1$, we obtain a tighter bound on vanishing, because the supply of exterior generators is limited:
\begin{prop}\label{gl2vanishing}
For $0<i<r(2p-3)$,
\[ H^i(GL_2\F_{p^r};\F_p)=H^i(B_2\F_{p^r};\F_p)
=H^i(\F_{p^r};\F_p)^{\F_{p^r}^\times}=0. \]
Also
\[ \dim H^{r(2p-3)}(GL_2\F_{p^r};\F_p)=1. \]
\end{prop}

\begin{proof}
It suffices to consider $H^*(\F_{p^r};\F_{p^r})^{\F_{p^r}^\times}$.
As above, any (nontrivial) $\F_{p^r}^\times$-invariant monomial
\[ z=\prod_{k=0}^{r-1}x_k^{a_k}y_k^{b_k}\in H^*(\F_{p^r};\F_{p^r}) \]
satisfies
\[ (p^r-1)\mid\sum_{k=0}^{r-1}p^k\left(a_k+b_k\right). \]
But now \[ \sum_{k=0}^{r-1}a_k\leq r, \]
so
\begin{align*}
\deg(z)&=\sum_{k=0}^{r-1}(a_k+2b_k) \\
&\geq \sum_{k=0}^{r-1}(2a_k+2b_k)-r \\
&\geq 2r(p-1)-r=r(2p-3).
\end{align*}
Therefore,
\[ H^i(\F_{p^r};\F_{p^r})^{\F_{p^r}^\times}=0 \]
when $0<i<r(2p-3)$.

Furthermore, if both inequalities achieve equality, then we must have all $a_k=1$ and all $b_k=p-2$, so there is precisely one such monomial, namely
\[ x_0\cdots x_{r-1}y_0^{p-2}\cdots y_{r-1}^{p-2}.\qedhere \]
\end{proof}

\pagebreak[3]
In the same way, we get a corresponding result for $SL_2\F_{p^r}$:
\begin{prop}\label{sl2vanishing}
Let $p$ be odd.  For $0<i<r(p-2)$,
\[ H^i(SL_2\F_{p^r};\F_p)=0. \]
Also
\[ \dim H^{r(p-2)}(SL_2\F_{p^r};\F_p)=1. \]
\end{prop}
The proof is exactly like that of Proposition~\ref{gl2vanishing}, except that now the action of the diagonal matrices on $U\iso\F_{p^r}$ factors through the subgroup $2\cdot\F_{p^r}^\times$ of perfect squares; hence we now get
\[ \frac{1}{2}(p^r-1)\mid\sum_{k=0}^{r-1}p^k\left(a_k+b_k\right). \]
In this case, the invariant monomial is
\[ x_0\cdots x_{r-1}y_0^\frac{p-3}{2}\cdots y_{r-1}^\frac{p-3}{2}. \]

\begin{rem}
In fact we could interpolate between the two results Lemma~\ref{edgegroup} and Proposition~\ref{gl2vanishing}, showing that
$H^i(\F_{p^r}^n\semidirect\F_{p^r}^\times;\F_p)=0$
for
\[ 0<i<r(2p-2-\min\{n,p-1\}). \]
Also notice that there is no difference between the two results in case $p=2$, since $r(2p-3)=r(p-1)=r$.  This is because there are no exterior generators.
\end{rem}

We pause here to note a fact that we will need later about the degree $r$ cohomology class for $p=2$.
\begin{lemma}\label{noperfectsquares}
The invariants $H^r(\F_{2^r};\F_2)^{\F_{2^r}^\times}$ intersect trivially with the subalgebra of perfect squares in $H^*(\F_{2^r};\F_2)$.
\end{lemma}
\begin{proof}
As shown in Proposition~\ref{gl2vanishing}, $H^r(GL_2\F_{2^r};\F_{2^r})$ is one-dimensional, generated by the class $x_0\cdots x_{r-1}$, which is not a perfect square in the polynomial algebra $H^*(\F_{2^r};\F_{2^r})=H^*(\F_{2^r};\F_2)\otimes\F_{2^r}$.  Neither is any scalar multiple of it, since squaring on $\F_{2^r}$ is surjective.  Let $\alpha\in H^r(GL_2\F_{2^r};\F_2)$ be the nonzero element.  If $\alpha$ were a perfect square, then so would be $\alpha\otimes1$, contradicting the previous observation.
\end{proof}

The analysis of this section yields a canonical basis for
$H^*(GL_2\F_{p^r};\F_{p^r})$, given by those monomials in the $x_k$'s (and, if $p$ is odd, the $y_k$'s), 
whose exponents satisfy
\[ (p^r-1)\mid\sum_{k=0}^{r-1}p^k(a_k+b_k). \]
In particular, it gives a complete description of the additive and multiplicative structure of the $\F_{p^r}$-cohomology, describing it as a subring of $H^*(\F_{p^r};\F_{p^r})$.
By the Universal Coefficient Theorem, the Poincar\'e series of the $\F_{p^r}$-cohomology agrees with the Poincar\'e series of the $\F_p$-cohomology.  Hence we may use this basis to answer questions about the dimensions of the $\F_p$-cohomology groups, as we implicitly did above.  However, great caution is needed in transferring any other information, because the above basis for $H^*(GL_2\F_{p^r};\F_{p^r})$ does not correspond to a basis of $H^*(GL_2\F_{p^r};\F_p)$.  In fact, as far as I know, the multiplicative structure of the $\F_p$-cohomology ring has not been determined.  However, some information about the ring structure does readily pass back and forth, as we shall see in section~\ref{gl2nilpotency}.

\section{Nilpotence in $H^*(GL_2\F_{p^r};\F_p)$}\label{gl2nilpotency}
We wish to study the nilpotent elements of $H^*(GL_2\F_{p^r};\F_p)$.  Of course, when $p=2$, there are no (nonzero) nilpotent elements, as this cohomology ring embeds in the cohomology of an elementary abelian 2-group, which is polynomial.  Accordingly, we now assume $p$ is odd.

We point out that the set of nilpotent elements in a graded-commutative algebra do form a two-sided ideal, which we call the nilradical.  (A proof is given in section~\ref{nilideal}.)

Using the basis for $H^*(GL_2\F_{p^r};\F_{p^r})$ described in the previous section, it is easy to describe the nilradical of the $\F_{p^r}$-cohomology ring.  Indeed, the nilradical of
\[ H^*(\F_{p^r};\F_{p^r})=\F_{p^r}\langle x_0,\dots,x_{r-1}\rangle
\otimes\F_{p^r}[y_0,\dots,y_{r-1}] \]
is the ideal $(x_0,\dots,x_{r-1})$, because the quotient by this ideal is an integral domain.  Hence the nilradical of $H^*(GL_2\F_{p^r};\F_{p^r})$
is spanned (as an $\F_{p^r}$-vector space) by those monomials
\[ \prod_k x_k^{a_k}y_k^{b_k} \]
with some $a_k\neq0$, and satisfying
\[ (p^r-1)\mid\sum_{k=0}^{r-1}p^k(a_k+b_k). \]

To descend this information about the cohomology with $\F_{p^r}$ coefficients to information about the cohomology with $\F_p$ coefficients, we have:\footnote{C.f.~Curtis and Reiner~\cite[Cor.~69.10]{CR} in the case where $A$ is finite-dimensional.}

\begin{prop}\label{separablenil}
Let $E$ be a finite separable field extension of $F$, and $A$ a commutative or graded-commutative $F$-algebra.  Then the inclusion
\[ \N(A)\otimes_F E\to\N(A\otimes_F E) \]
is an isomorphism.  (Here $\N$ represents the ideal of nilpotent elements.)
\end{prop}

We remark that this fails for inseparable field extensions: consider $E\otimes_F E$ where $F=\F_p(t)$ and $E=F(\sqrt[p]{t})$.  Also, if $A$ is not (graded) commutative, it does not make sense, as $\N(A)$ can fail to be a subspace.

The proof uses the following fact belonging to the theory of Galois descent:
\begin{lemma}\label{descent}
Let $E$ be a Galois extension of $F$ with Galois group $G=\gal(E:F)$.  Let $V$ be a vector space over $F$, and 
\[ W\subset V\otimes_F E \]
an $E$-linear subspace.  If $W$ is preserved by $G$ (acting on $E$), then the natural map
\[ (W\cap V)\otimes_F E\to W \]
is an isomorphism.  I.e., $W$ is spanned by its elements belonging to $V$.
\end{lemma}
\noindent


\begin{proof}[Proof of Proposition~\ref{separablenil}]
In case $E$ is a Galois extension of $F$, we apply Lemma~\ref{descent} to $W=\N(A\otimes_F E)$, a subspace which is preserved under all ring automorphisms.  This gives the result, since $\N(A)=W\cap A$.

Now suppose $E$ is merely finite and separable over $F$; then it is contained in a Galois extension $K$ of $F$.  Hence by the Galois case,
\[ \N(A)\otimes_F K=\N(A\otimes_F K). \]
Since \[ (\N(A)\otimes_F K)\cap(A\otimes_F E)=\N(A)\otimes_F E, \]
this gives the desired result.
\end{proof}


In our case, Proposition~\ref{separablenil} gives
\begin{cor}\label{gl2nilp}
\[ \N(H^*(GL_2\F_{p^r};\F_p))\otimes\F_{p^r}=\N(H^*(GL_2\F_{p^r};\F_{p^r})). \]
In particular, the Poincar\'e series of the nilradical in $\F_{p^r}$-cohomology agrees with the Poincar\'e series of the nilradical in $\F_p$-cohomology.
\end{cor}

For example, 
\begin{prop}\label{gl2nonnilp}
The lowest degree homogeneous nonnilpotent element of $H^*(GL_2\F_{p^r};\F_p)$ occurs in degree $r(2p-2)$.
\end{prop}

\begin{proof}
By Corollary~\ref{gl2nilp}, it suffices to find the lowest nonnilpotent element of $H^*(GL_2\F_{p^r};\F_{p^r})$.  Just as in the proof of Proposition~\ref{gl2vanishing}, we need to find the lowest degree (nontrivial)  $\F_{p^r}^\times$-invariant monomial
\[ z=\prod_{k=0}^{r-1}x_k^{a_k}y_k^{b_k}\in H^*(\F_{p^r};\F_{p^r}) \]
which is not nilpotent.  This means that all $a_k=0$.  Hence
\begin{align*}
\deg(z)&=\sum_{k=0}^{r-1}2b_k \\
&\geq 2r(p-1),
\end{align*}
using Lemma~\ref{Quillenlemma} as before.  Of course, the nonnilpotent monomial is
\[ z=y_0^{p-1}\cdots y_{r-1}^{p-1}. \qedhere \]
\end{proof}


\section{Nilpotence in graded commutative algebras}\label{nilideal}
In this section, we verify that the nilradical of a graded commutative algebra over a field is indeed an ideal.\footnote{Thanks to Riley Casper and John Palmieri for suggestions.}  (In section~\ref{gl2nilpotency}, we used the fact that it is a subspace.)

Let $R$ be a ($\bbZ$-)graded commutative algebra over a field $F$, of odd characteristic.  (If the characteristic of $F$ were $2$, then $R$ would be commutative, rendering this section unnecessary.)
Define $\N(R)$ to be the set of nilpotent elements.  We show that $\N(R)$ forms an ideal.  In a general noncommutative ring, it could fail to be closed under addition: e.g.~consider $F\{x,y\}/(x^2,y^2)$.

For $x\in R$, write
\[ x=x_o+x_e, \]
where $x_o$ is a sum of odd-degree homogeneous elements and $x_e$ a sum of even-degree homogeneous elements.
\pagebreak[3]
\begin{lemma}
Let $x\in R$.
\begin{enumerate}[(a)]
\item $x_o^2=0$;
\item $x$ is nilpotent if and only if $x_e$ is.
\end{enumerate}
\end{lemma}
\begin{proof}
\begin{enumerate}[(a)]
\item Clear because odd degree elements anticommute, and $2\in F$ is a unit.
\item One direction is clear since $x_o$ and $x_e$ are commuting nilpotents.  Suppose now that $x$ is nilpotent.  Let $x'=x_e-x_o$.  We have
\[ xx'=x'x=x_e^2. \]
Since $x$ is nilpotent, it follows that $x_e^2$ is nilpotent, so $x_e$ is. \qedhere
\end{enumerate}
\end{proof}

\begin{lemma}\label{nilpotentcomponents}
An element of $R$ is nilpotent if and only if its homogeneous components are all nilpotent.
\end{lemma}
\begin{proof}
If the components of $x\in R$ are all nilpotent, then the sum of even-degree components is nilpotent, because they commute; hence by the lemma $x$ is nilpotent.
On the other hand, suppose $x$ is nilpotent. We need to show that its even degree components are nilpotent (as the odd-degree components are nilpotent always).  Using the lemma, it suffices to assume that $x$ is a sum of only even-degree components.  Apply induction on the degree of the highest degree component $z$.  If $x^n=0$, then $z^n=0$ also, so $z$ is nilpotent.  But, since $x$ and $z$ commute, $x-z$ is nilpotent also, so the induction hypothesis applies to show that the remaining components of $x$ are nilpotent.
\end{proof}

\begin{prop}
The set $\N(R)\subset R$ of nilpotent elements forms a (two-sided) ideal.
\end{prop}

\begin{proof}
Let $x,y\in R$ be nilpotent.  By Lemma~\ref{nilpotentcomponents}, to see that $x+y$ is nilpotent, it suffices to check that its homogeneous components are.  Hence we are reduced to the case where $x$ and $y$ are homogeneous: clear because they commute or anticommute.

Now let $x,z\in R$ with $x$ nilpotent.  To show that $xz$ and $zx$ are nilpotent, it suffices to assume both $x$ and $z$ are homogeneous, as we now know that $\N(R)$ is closed under addition (and the components of $x$ are nilpotent).  But then they either commute or anticommute, showing their product is nilpotent.
\end{proof}
In particular, $\N(R)$ is an $F$-linear subspace of $R$.

\chapter{Nonvanishing cohomology classes on finite groups of Lie type with Coxeter number at most \lowercase{$p$}}\label{lietypech}
In this chapter, we prove that the degree $r(2p-3)$ cohomology of any (untwisted) finite group of Lie type over $\F_{p^r}$, with coefficients in characteristic $p$, is nonzero as long as its Coxeter number is at most $p$.  We do this by providing a simple explicit construction of a nonzero element.

Most of the material of this chapter has previously appeared~\cite{lietype} in the Journal of Pure and Applied Algebra.\footnote{Sprehn, D.: Nonvanishing cohomology classes on finite groups of Lie type with Coxeter number at most $p$,
{\it Journal of Pure and Applied Algebra} 219, (2015), 2396-2404.  doi:10.1016/j.jpaa.2014.09.006}

\section{Introduction}
We aim to construct a mod-$p$ cohomology class on a finite group $G$ of Lie type over the finite field $\F_{p^r}$, nonzero if the Coxeter number of $G$ is at most $p$.
More specifically, we will provide an embedding of the invariants
\[ H^*(\F_{p^r};\F_p)^{\F_{p^r}^\times} \iso H^*(GL_2\F_{p^r};\F_p) \]
into $H^*(G;\F_p)$.  Since these invariants are nonzero (see section~\ref{gl2}) in degree $r(2p-3)$, this shows $H^{r(2p-3)}(G;\F_p)\neq0$.

In particular, for $G=GL_n\F_{p^r}$, we have:
\begin{thm}\label{glnnonvanishing}
Suppose $2\leq n\leq p$.  Then
\[ H^{r(2p-3)}(GL_n\F_{p^r};\F_p)\neq0. \]
\end{thm}
\noindent
This extends a result of Bendel, Nakano, and Pillen~\cite{BNP} (which was valid for
\linebreak $n\leq p-2$).  
In chapter~\ref{dimensionch}, we will show that, at least in the cases $p=2$ or $r=1$, the group $H^{r(2p-3)}(GL_n\F_{p^r};\F_p)$ is one-dimensional, so that the nonzero element we construct in this chapter will in fact generate.  Also we will show in those cases that $H^{r(2p-3)}(GL_n\F_{p^r};\F_p)=0$ for $n>p$, so that the assumption $n\leq p$ in the above result cannot be improved further.

Friedlander and Parshall~\cite{FP} proved that the cohomology of $G=GL_n\F_{p^r}$ vanishes below degree $r(2p-3)$, i.e.~this is the lowest degree in which nonzero mod-$p$ cohomology can occur.  So our result in this case can be interpreted as saying that $r(2p-3)$ is a sharp bound for vanishing of $H^*(GL_n\F_{p^r};\F_p)$.  Bendel, Nakano, and Pillen~\cite{BNP,BNP2} have made progress on the problem of finding the sharp vanishing bounds for finite groups of Lie type.  In general, the bound $r(2p-3)$ is not sharp, but it is sharp in many cases, including the simply-laced, adjoint type groups of Coxeter number less than $p$, and the simply-connected groups of types $A_n$ ($n\geq 4$), $E_6$, $E_7$, and $E_8$, with sufficiently large $p$.
Furthermore, in simply-connected type $A_{n-1}$ with $5\leq n\leq (p-1)/2$ they showed the lowest cohomology group $H^{r(2p-3)}(SL_n\F_{p^r};\F_p)$ to be one-dimensional, so our nonvanishing class is a generator in these cases.  It follows that the same result holds for $GL_n\F_{p^r}$.  See also Theorems~\ref{lowestcoh2} and ~\ref{onedimensional}.  Also, in simply-connected type $C_n$ (i.e.~$Sp(2n,\F_q)$) with $n<p/2$, they showed that the lowest nonvanishing cohomology $H^{r(p-2)}(Sp(2n,\F_q);\F_p)$ is one-dimensional, so the nonvanishing class we shall construct there is also a generator.

In the case $r=1$, the invariants $H^*(\F_p;\F_p)^{\F_p^\times}$
contain nonvanishing classes in degrees $2p-3$ and $2p-2$, connected by the Bockstein homomorphism.  Hence, for $G$ with Coxeter number at most $p$, $H^*(G;\F_p)$ also contains such classes.  This answers a conjecture of Barbu~\cite{Barbu} who constructed a class in $H^{2p-2}(GL_n\F_p;\F_p)$ for $n\leq p$ and conjectured it to be a Bockstein.  We remark that this class in degree $2p-2$  may also be viewed as a Chern class of the permutation representation of $GL_n\F_p$ on the set $\F_p^n-\{0\}$; see section~\ref{chernclass}.  The conjecture may also be economically verified as a result of this observation, although this approach does not give any information in the case $r>1$.

\section{Notation for finite groups of Lie type}\label{notationlietype}
Fix $q=p^r$.  We investigate the cohomology of finite groups of Lie type over $\F_q$: those groups which arise as the set of fixed-points of a $q$-Frobenius map $F$ acting on a connected reductive algebraic group.

Accordingly, let $\bar G$ be a connected reductive algebraic group over $\bar\F_q$.  Let $F:\bar G\to\bar G$ be a {\bf $q$-Frobenius map}.  That is,\footnote{Other authors refer to these as ``standard'' or ``untwisted'' Frobenius maps.} a map obtained from the endomorphism $F_q((a_{ij}))=(a_{ij}^q)$ of $GL_N(\bar\F_q)$ via some closed embedding $\bar G\leq GL_N(\bar\F_q)$ whose image is preserved by $F_q$.  Any such map is surjective (since $\bar G$ is connected).  The (finite) fixed-point group $\bar G^F$ of a $q$-Frobenius map is called a {\bf finite group of Lie type} over $\F_q$.

For discussion of the following notions, see Carter~\cite{Carter}.
Fix a choice of maximal torus and Borel subgroup $\bar T\leq\bar B\leq\bar G$ which are preserved by $F$.  Let $\bar U$ be the unipotent radical of $\bar B$.
Write $G=\bar G^F$, and similarly $B=\bar B^F$, $T=\bar T^F$, $U=\bar U^F$.
We remark that $B=U\semidirect T$ and $N=N_{\bar G}(\bar T)^F$ form a split BN-pair of characteristic $p$ in $G$, and that $U\leq G$ is $p$-Sylow subgroup.

Let $W=N/T$ be the Weyl group of $G$, $S$ be its set of Coxeter generators, and $\Phi$ be the associated root system with base (determined by $B$) 
\[ \Delta=\set{\alpha_s}{s\in S} \]
and positive roots $\Phi^+$.  For $\alpha\in\Phi$, denote by $X_\alpha$ the corresponding root subgroup of $G$.  Thus a root subgroup $X_\alpha$ is contained in $U$ if and only if $\alpha$ is positive, and $U$ is generated by the positive root subgroups.  In particular, for $s\in S$, we shall abbreviate
\[ X_s=X_{\alpha_s}=U\cap s\inv w_0\inv Uw_0s \]
(where $w_0\in W$ is the longest element).  Each root subgroup is isomorphic to $\F_q$, and is normalized by $T$,
with the action being $\F_q$-linear.
The root subgroups $X_\alpha$ satisfy the ``Chevalley commutator relations:''
\[ [X_\alpha,X_\beta]\leq\langle X_{i\alpha+j\beta} | i,j>0 \rangle \]
for all $\alpha,\beta\in\Phi$ with $\alpha\neq\pm\beta$.
(Note that $X_\gamma$ should be interpreted as trivial when $\gamma\notin\Phi$.)

Also define
\[ U_s=U\cap s\inv Us=\prod_{\alpha\in\Phi^+-\{\alpha_s\}}X_\alpha, \]
a normal subgroup of $U$ by the commutator relations (since $U$ is generated by the positive root subgroups).
The composite
\[ X_s\to U\to U/U_s \]
is an isomorphism.

Lastly, each $\alpha\in\Phi^+$ can be written
\[ \alpha=\sum_{s\in S} c_s\alpha_s \]
with the $c_i$ nonnegative integers; we define the height of $\alpha$ to be
\[ \hgt(\alpha) = \sum_{s\in S}c_s \in\bbZ^+ \]
and the Coxeter number of $G$ to be $h=\max\set{\hgt(\alpha)+1}{\alpha\in\Phi^+}$.

The reader may wish to keep in mind the case $G=GL_n\F_{p^r}$, in which $n$ is the Coxeter number, $B$ (resp.~$U$) the group of invertible (resp.~unipotent) upper triangular matrices, $T$ the diagonal matrices,
$X_s$ the subgroups
\[ X_k=\set{(a_{ij})\in GL_n\F_{p^r}}{a_{ij}=\delta_{ij}\mbox{ unless }(i,j)=(k,k+1)}, \]
and $U_s$ the subgroups
\[ U_k=\set{(a_{ij})\in U}{a_{k,k+1}=0}. \]

\section{Commuting regular unipotents}
Our goal is to provide an injection (Theorem~\ref{injection}) from the $T$-invariant cohomology of a root subgroup $X_s$ to the cohomology of $G$, by composing the pullback to $H^*(B;\F_p)$ with the transfer map up to $H^*(G;\F_p)$.  We will verify that the composite is injective by then restricting to a certain elementary abelian $p$-subgroup consisting of regular unipotent elements.  This section is devoted to showing the existence of such a subgroup (Corollary~\ref{ptorus}).

We remark that, in the case $G=GL_n\F_{p^r}$, this section may be bypassed and Corollary~\ref{ptorus} proved by constructing the required subgroup directly as the (elementary abelian) subgroup generated by the matrices
\[ I+\lambda \begin{bmatrix}
0 & 1 & & \\ & \ddots & \ddots & \\ & & & 1 \\ & & & 0
\end{bmatrix}, \]
where $\lambda$ ranges over a choice of basis for $\F_{p^r}$ over $\F_p$.

\subsection{Regular unipotents}
\begin{defn}
For our purposes, an element $x\in G$ is unipotent if it lies in a conjugate of $U$, or equivalently, its order is a power of $p$.  Furthermore, $x$ is {\bf regular unipotent} in $G$ if its action on $G/B$ has a unique fixed point.  Since $N_G(B)=B$, this is equivalent to lying in a unique conjugate of $B$.
\end{defn}
There is of course a corresponding notion for an element of the algebraic group: a unipotent element $x\in\bar G$ is {\bf regular unipotent} in $\bar G$ if its action on $\bar G/\bar B$ has a unique fixed point, or equivalently, $x$ lies in a unique Borel subgroup.

For an element $x\in G$ of the finite group, then, there are two notions of being regular unipotent.  However, they coincide.  This follows from bijectivity~\cite[sec.~1.17]{Carter} of the map
\[ G/B=\bar G^F/\bar B^F\to (\bar G/\bar B)^F, \]
together with the observation that, when a regular unipotent in $\bar G$ is fixed by $F$, the (unique) Borel subgroup containing it is preserved by $F$.

The set of regular unipotent elements in $G$ is clearly preserved by conjugacy in $G$.  In fact, the above discussion shows that it is preserved by conjugacy in $\bar G$.

There are other ways to characterize the regular unipotents: see~\cite[p.~29; Prop.~5.1.3]{Carter}.

Recall that, for each Coxeter generator $s\in S$, we defined $U_s=U\cap s\inv Us$, which is normal in $U$.
One can show: 
\begin{lemma}\label{regularunips}
Let $x\in U$.
\begin{enumerate}[(a)]
\item $x$ is regular unipotent if and only if $x\notin U_s$ for all $s\in S$;
\item There exist regular unipotent elements in $G$.
\end{enumerate}
\end{lemma}
\noindent These results are well-known; see e.g.~\cite[Prop.~5.1.3; Prop.~5.1.7(a)]{Carter}.  We give here a proof which differs from Carter's by emphasizing the BN-pair structure of $G$ rather than the geometry of the ambient algebraic group $\bar G$.  We will, however, defer to Carter on the basic theory of BN-pairs.\footnote{We warn the reader that our notation differs from that which Carter tabulates in~\cite[p.~50]{Carter}.}
The finite group $G$ indeed satisfies the axioms for a group with split BN-pair as in~\cite[sec.~1.18]{Carter}.

Write $B_v=B\cap v\inv Bv$, $U_v=U\cap v\inv Uv\leq B_v$, and $X_s=U\cap s\inv w_0\inv Uw_0s$, as before.
We will need to use four facts about split BN-pairs:

\begin{lemma}\label{bnfacts}
Let $(G,B,N,U)$ be a split BN-pair with Coxeter system $(W,S)$.
Let $v\in W$, and let $s,s'\in S$ be distinct.
\begin{enumerate}[(a)]
\item Bruhat Decomposition: $G=\bigcup_{w\in W}B\bar wB=\bigcup_{w\in W}B\bar wU$.
\item $B_{sv}\leq B_v$ if $\len(sv)=\len(v)+1$;
\item $X_s\cap U_s=1$;
\item $X_s\leq U_{s'}$.
\end{enumerate}
\end{lemma}

\begin{proof}
\begin{enumerate}[(a)]
\item Implied by~\cite[Thm.~2.5.14]{Carter} or~\cite[Prop.~2.1.1]{Carter}.
\item This is~\cite[Prop.~2.5.4]{Carter}.
\item This is~\cite[Prop.~2.5.11]{Carter}.
\item First, by~\cite[Prop.~2.2.6]{Carter} and~\cite[Prop.~2.5.7(i)]{Carter} (with $w=s'$), we have $B_{w_0s}\leq B_{s'}$.  Since for all $v\in W$, $U_v\leq B_v$ is a normal Sylow $p$-subgroup, this shows $X_s=U_{w_0s}\leq U_{s'}$. \qedhere
\end{enumerate}
\end{proof}

\begin{proof}[Proof of Lemma~\ref{regularunips}]\hfill
\begin{enumerate}[(a)]
\item The ``only if'' direction is clear from the definition of regular unipotent elements, since $U_s\leq B\cap s\inv Bs$, the intersection of two distinct\footnote{By one of the BN-pair axioms:~\cite[p.42,~axiom~(iv)]{Carter}} conjugates of $B$.  For the converse, suppose $x\in U$ is not regular unipotent.  Then there exists $g\in G$ such that $x\in g\inv Bg\neq B$.  Using the Bruhat decomposition,
we may write $g=b\bar wu$, with $u\in U$, $b\in B$, and $\bar w\in N$ a representative for $w\in W$.  We may clearly assume $b=1$.  We wish to replace $x$ with $uxu\inv$ and thereby assume $u=1$.  This does not affect whether our element is regular unipotent (since regularness is constant on conjugacy classes); neither does it affect membership in $U_s\normal U$.  Now we have reduced to the case $g=\bar w$, that is,
$x\in B_w=B\cap w\inv Bw$.  Note that $w\neq1$ since $w\inv Bw\neq B$.  It will therefore suffice to check that $B_w\leq B_s$ for some $s\in S$.
By induction on $\len(w)$, this follows from Lemma~\ref{bnfacts}(b).
\item Choose some ordering $S=\{s_1,\dots, s_n\}$.  For each $i$, pick a nontrivial element
\[ x_i\in X_{s_i} = U_{w_0s_i}. \]
Then set $x=x_1\cdots x_n\in U$.  By Lemma~\ref{bnfacts}(c,d), we know $x_j\in U_{s_i}$ whenever $i\neq j$, and that $x_i\notin U_{s_i}$.  Therefore $x\notin U_{s_i}$ either, as desired. \qedhere
\end{enumerate}
\end{proof}

\subsection{Elementary abelian overgroups}
We wish to show (Corollary~\ref{ptorus}) that, when the Coxeter number of $G$ is at most $p$, there exists an elementary abelian $p$-subgroup of rank $r$, in which every nontrivial element is regular unipotent.  (Recall $q=p^r$.)  This follows from two facts:
\begin{enumerate}
\item When the Coxeter number of $G$ is at most $p$, every nontrivial unipotent element has order $p$.
\item When $x\in G$ has order $p$, its $\bar G$-conjugacy class (plus the identity) contains an elementary abelian $p$-subgroup of rank $r$.
\end{enumerate}

The former statement is no doubt familiar, but we include an elementary proof at the end of this chapter (Proposition~\ref{exponent}).  We shall show the latter using Testerman's Theorem on $A_1$-overgroups:

\begin{thm}[Testerman~\cite{Testerman}]
Let $\bar G$ be a semisimple algebraic group over an algebraically closed field $k$ of nonzero characteristic $p$.  Assume $p$ is a good prime\footnote{This means that $p$ does not divide the coefficients of any root $\alpha\in\Phi^+$ when expressed in the simple root basis.}
for $\bar G$.  Let $\sigma$ be a surjective endomorphism of $\bar G$ with finite fixed-point subgroup.  Let $u\in\bar G^\sigma$ with $u^p=1$.  Then there exists a closed connected subgroup $X$ of $\bar G$ with $\sigma(X)\leq X$, $u\in X$, and $X$ isomorphic to $SL_2(k)$ or $PSL_2(k)$.
\end{thm}
We begin by extending Testerman's theorem to reductive groups.
\begin{thm}\label{Testerman}
Let $\bar G$ be a connected reductive algebraic group over an algebraically closed field $k$ of nonzero characteristic $p$.  Assume $p$ is a good prime for $\bar G$.  Let $\sigma$ be a surjective endomorphism of $\bar G$ with finite fixed-point subgroup.  Let $u\in\bar G^\sigma$ with $u^p=1$.  Then there exists a closed connected subgroup $X$ of $\bar G$ with $\sigma(X)\leq X$, $u\in X$, and $X$ isomorphic to $SL_2(k)$ or $PSL_2(k)$.
\end{thm}

\begin{proof}
Since $\bar G$ is connected reductive, its derived subgroup $\bar G'$ is closed~\cite[sec.~2.3]{Borel}, connected and semisimple; it is preserved by $\sigma$.  We claim the restriction of $\sigma$ to $\bar G'$ remains surjective.  This is because its image is a closed subgroup of $\bar G'$ with the same dimension.  Furthermore, $\bar G'$ contains all the unipotent elements of $\bar G$, including $u$; it also has the same root system.  Therefore, Testerman's theorem applies, producing the desired subgroup $X$.
\end{proof}

Using this, we will obtain:
\begin{prop}\label{conj}
Let $\bar G$ be a connected reductive group with a $p^r$-Frobenius map $F$.  Assume $p$ is a good prime for $\bar G$.  If $u\in\bar G^F$ has order $p$, then there exists an elementary abelian $p$-subgroup of rank $r$ in $\bar G^F$, containing $u$, all of whose nontrivial elements are conjugate in $\bar G$.
\end{prop}
For the finite groups of Lie type $\bar G^F$ with Coxeter number at most $p$, this proposition applies (by Proposition~\ref{exponent}) to the regular unipotent elements:
\begin{cor}\label{ptorus}
Let $G$ be a nontrivial\footnote{I.e.~positive rank.} group of Lie type over $\F_{p^r}$ with Coxeter number at most $p$.  There exists an elementary abelian $p$-subgroup of rank $r$ in $G$, all of whose nontrivial elements are regular unipotent.
\end{cor}
We point out that the assumption of $p$ being a good prime is automatic when $G$ has Coxeter number at most $p$.

In order to prove the proposition, we need the following lemma on Frobenius maps of a one-dimensional additive group.  It may be viewed as showing how to recover $q$ from a $q$-Frobenius map.  Specifically, it is not possible for a $q$-Frobenius map and a $q'$-Frobenius map to agree on a reductive group of positive rank unless $q=q'$ (apply the lemma to a root subgroup).

\begin{lemma}\label{Ga}
Let $k$ be an algebraically closed field of characteristic $p$.
\nopagebreak[2]
\begin{enumerate}[(a)]
\item Let $\phi:k[x]\to k[x]$ be determined by $\phi(x)=cx^{p^d}$ for some $c\in k^\times$ and $d\geq0$.  If $p(x)$ is any nonconstant polynomial such that $\phi(p(x))=p(x)^{p^e}$, then $d=e$.
\item Let $X\leq GL_N(k)$ be a (closed) subgroup isomorphic to the additive group $\mathbb{G}_a(k)$, which is preserved by the standard Frobenius map $F_{p^r}$.  Then (as finite groups)
\[ X^{F_{p^r}}\iso C_p^r. \]
\end{enumerate}
\end{lemma}

\begin{proof}
The first statement is easily checked by expanding and matching nonzero terms.
The second will follow after observing that any injective endomorphism of $\mathbb{G}_a$ must have the form $u\mapsto cu^{p^d}$ for some $c\in k^\times$ and $d\geq0$.  This is because
an endomorphism of $\Ga$ is given by a single polynomial function, which must be $\F_p$-linear, i.e.~has the form
\[ u\mapsto\sum_{i=0}^k a_i u^{p^i}; \]
Such a map is injective if it has trivial kernel, which occurs if and only if precisely one $a_i$ is nonzero.  So it must take the form $F:u\mapsto cu^{p^d}$.  The invariants of such a map are
\[ \Ga^F=\frac{1}{b}\F_{p^d}, \]
where $b^{p^d-1}=c$.  In particular,
\[ \Ga^F\iso C_p^d \]
as finite groups.  It remains only to show that $d=r$, for which we apply part (a) to the coordinate functions of the embedding $X\leq GL_N(k)$, at least one of which is nonconstant.
\end{proof}

\begin{proof}[Proof of Proposition~\ref{conj}]
A $p^r$-Frobenius map $F$ is surjective with finitely many fixed points.  Then by Theorem~\ref{Testerman}, we know that $u$ is contained in a closed $F$-invariant subgroup isomorphic to either $SL_2(\bar\F_p)$ or $PSL_2(\bar\F_p)$.  In these groups, every element of order $p$ is conjugate.  Furthermore, the unipotent radical $U$ of an $F$-invariant Borel subgroup is isomorphic to $\mathbb{G}_a$, hence by Lemma~\ref{Ga} its $F$-invariants form an elementary abelian $p$-subgroup of rank $r$, all of whose nontrivial elements are conjugate in $\bar G$.
\end{proof}

\section{Cohomology}\label{cohomology}
Now we have all the ingredients in place for our theorem on cohomology of $G$.
Recall that $q=p^r$; $G,T,B,U,S$ are, respectively, a finite group of Lie type over $\F_q$, a maximal torus, a Borel containing $T$, its unipotent radical, and the set of Coxeter generators for the Weyl group of $G$.  For each $s\in S$, $X_s$ is the corresponding simple root subgroup, which is isomorphic to $\F_q$.

\begin{thm}\label{injection}
Let $G$ be a finite group of Lie type over $\F_q$ having Coxeter number at most $p$.  For each $s\in S$, there is an injective graded vector space map from $H^*(X_s;\F_p)^T$ to $H^*(G;\F_p)$.  It is given by composing the pullback $H^*(X_s\semidirect T;\F_p)\to H^*(B;\F_p)$ with transfer $H^*(B;\F_p)\to H^*(G;\F_p)$.  Moreover, it is a module homomorphism over the Steenrod algebra.
\end{thm}
We remark that the last sentence follows because transfer maps commute with the Steenrod operations~\cite{Evens}, as do maps induced on cohomology by group homomorphisms.

\begin{proof}
Let $A$ be an elementary abelian $p$-subgroup of rank $r$ consisting of regular unipotent elements, whose existence is guaranteed by Corollary~\ref{ptorus}; we may assume $A\leq U$.  Fix an $s\in S$, and consider the composition
\[ A\to U\to U/U_s. \]
The composition is injective, because no regular unipotent lies in $U_s$.  Hence it is an isomorphism because $U/U_s\iso X_s\iso\F_q\iso A$ as groups.  We shall use this fact below.

Recall that $U_s\normal B$.  Because $B$ is generated by $T$, $U_s$, and $X_s$, the quotient is
\[ B/U_s=X_s\semidirect T. \]
For the remainder of this proof, all cohomology is with $\F_p$ coefficients, which are suppressed from the notation.
Now consider the map on cohomology given by
\[ H^*(X_s)^T=H^*(B/U_s)\to H^*(B)\xrightarrow{\operatorname{tr}}
H^*(G)\xrightarrow{i^*} H^*(A). \]
The first map is induced by the quotient homomorphism, the second is the transfer map, and the third restriction.  We wish to show that this composition is injective, from which the claims in the theorem follow.
The double coset formula expresses the composition $i^*\circ\operatorname{tr}$ as a sum indexed over $A\backslash G/B$.  Since $A$ is elementary abelian, all of the terms vanish except those corresponding to the fixed points of $A$ on $G/B$.  But as $A$ contains regular unipotents, $B$ is the only such fixed point.  Therefore the composition $i^*\circ\operatorname{tr}$ equals the restriction map $H^*(B)\to H^*(A)$.
Hence the above composition is
\[ H^*(B/U_s)\to H^*(B)\to H^*(A), \]
which in turn equals
\[ H^*(B/U_s)\to H^*(U/U_s)\to H^*(U)\to H^*(A). \]
But the first map is injective since $p$ does not divide the index, while the composition of the other two is an isomorphism, as remarked at the start of this proof.  Hence the map in question is injective as desired.
\end{proof}

\begin{cor}\label{nonvanishing}
Let $G$ be a finite group of Lie type over $\F_q$ having Coxeter number at most $p$, which is nontrivial (of positive rank).  Then $H^*(G;\F_p)$ contains $H^*(GL_2\F_q;\F_p)$, as a (graded) submodule over the Steenrod algebra. In particular,
\[ H^{r(2p-3)}(G;\F_p)\neq0. \]
\end{cor}
\begin{proof}
Pick any $s\in S$.  It suffices to express
$H^*(GL_2\F_q;\F_p) = H^*(\F_q;\F_p)^{\F_q^\times}$ as such a submodule of $H^*(X_s;\F_p)^T$.  It is, since $X_s\iso\F_q$ with $T$ acting $\F_q$-linearly.  Recall that $H^{r(2p-3)}(GL_2\F_q;\F_p)$ is nontrivial---Lemma~\ref{gl2vanishing}.
\end{proof}

In one case, namely the simply-connected groups in type $C$, there is a choice of $s$ for which
\[ \im(T\to\aut(X_s))<\F_q^\times \]
is a proper subgroup; that is, $T$ does not act transitively on the nonzero elements of $X_s$.  Here we get a stronger result than that of Corollary~\ref{nonvanishing}.
\begin{cor}\label{sp2n}
Suppose $2n\leq p$.  Then $H^*(Sp(2n,\F_q);\F_p)$ contains
$H^*(SL_2\F_q;\F_p)$, as a (graded) submodule over the Steenrod algebra. In particular,
\[ H^{r(p-2)}(Sp(2n,\F_q);\F_p)\neq0. \]
\end{cor}
\begin{proof}
We may assume $p$ is odd.  Let $\alpha_s$ be the long simple root.  Then the action of $T$ on $X_s\iso\F_q$ factors through
$2\cdot\F_q^\times$, the subgroup of squares.  (This may be computed directly; or, see the next section.)  The invariants of $H^*(\F_q;\F_p)$ by this group are $H^*(SL_2\F_q;\F_p)$, and are nonvanishing
in degree $r(p-2)$: see Proposition~\ref{sl2vanishing}.
\end{proof}

Our main results for this chapter are now complete.
In the next section (independent from the preceding results), we show that the groups $Sp(2n,\F_q)$ are essentially the only case in which the phenomenon of Corollary~\ref{sp2n} occurs; that is, in all other cases the action map $T\to\F_q^\times$ on every root subgroup is surjective.

\section{Root action surjectivity}\label{rootaction}
Quillen~\cite[sec.~11]{QuillenK} remarked that
``for $SL_2$ and $p$ odd...the map $T\to k^\times$ given by the unique root of $SL_2$ has image of index 2.''  We have used the fact that this is true more generally\footnote{Of course, $SL_2=Sp_2$.} for $Sp_{2n}$ and its long roots.  In this section, we explore this observation, showing that $Sp_{2n}$ is the only case in which it occurs.

Let $\bar G$ be a connected reductive group and $F:\bar G\to \bar G$ a $q$-Frobenius map.
Let $\bar T\leq\bar G$ be a maximal torus split over $\F_q$ and
\[ \Phi\subset\chi(\bar T)=\hom(\bar T,\bar\F_q^\times) \]
be its root system.  Let $\bar\alpha\in\Phi$.  Consider the induced map on finite groups
\[ \alpha:T\to\F_q^\times \]
describing the action of $T$ on $X_\alpha$.  In this section, we investigate the question of when the invariants $H^*(X_\alpha;\F_p)^T$ of $T$ acting on a root subgroup can be bigger than the ``universal invariants''
$H^*(\F_q;\F_p)^{\F_q^\times}$,
which would require that $\alpha$ fails to be surjective.

First of all, we consider $\bar T$ in isolation.  Let $\bar T$ be an algebraic torus split over $\F_q$.  The following lemma, which we will check in coordinates, shows when an arbitrary character fails to induce a surjective map on the finite groups.

\begin{lemma}
Let $n|(q-1)$ and $\bar\alpha\in\chi(\bar T)$.
Then \[ n|[\F_q^\times:\im(\alpha)] \]
if and only if $\bar\alpha$ is divisible\footnote{In general, we say that an element $x$ of an abelian group $A$ is divisible by $n\in\bbZ$ if there exists $y\in A$ with $x=n\cdot y$.  We say that $x$ is divisible if it is divisible by some $n>1$.} by $n$ in $\chi(\bar T)$.
\end{lemma}

\begin{proof}
Pick an (algebraic) isomorphism identifying $T$ with $({\bar\F_q}^\times)^k$.  Then $\bar\alpha:\bar T\to{\bar\F_q}^\times$ takes the form
\[ \bar\alpha(t_1,\dots,t_k)=t_1^{a_1}\cdots t_k^{a_k} \]
for some $a_i\in\bbZ$.  But then it is clear that
\[ \im(\alpha)=\langle a_1\dots,a_k\rangle\leq\bbZ/(q-1)\iso\F_q^\times. \]
In particular $\im(\alpha)$ is contained in $\langle n\rangle<\bbZ/(q-1)$ if and only if $n|a_i$ for all $i$, in which case $\alpha$ is divisible by $n$.
\end{proof}

Now return to the situation where $\bar T$ is a maximal $\F_q$-split torus in a connected reductive group $\bar G$, of which $\bar\alpha$ is a root.  The lemma says that $\alpha$ is surjective unless $\bar\alpha$ is divisible in $\chi(\bar T)$ by some integer dividing $q-1$.  Since no root may be divisible (by any integer greater than one) in the root lattice, this immediately proves:

\begin{cor}
If $\bar G$ has adjoint type (i.e.~$\chi(\bar T)=\bbZ\Phi$), then $\alpha$ is surjective for every root $\bar\alpha$.
\end{cor}

Furthermore, the lemma shows that surjectivity of $\alpha$ can never fail unless $\bar\alpha$ is divisible in the (larger) weight lattice.  But this weaker condition still never occurs except in one case:

\begin{lemma}
Let $\Phi$ be a (reduced, crystallographic) root system, and $\Lambda\geq\bbZ\Phi$ the associated weight lattice.  Then no root $\bar\alpha\in\Phi$ is divisible in $\Lambda$, unless (an irreducible component of) $\Phi$ has type $C_n$, in which case a long root $\bar\alpha$ is divisible by 2.
\end{lemma}

\begin{proof}
Since the Weyl group acts on $\Lambda$ by automorphisms, it suffices to check the proposition for simple roots $\alpha_s$, $s\in S$.
So, suppose that
\[ \alpha_s=m\cdot\beta \]
for some $m>1$ and $\beta\in\Lambda$.  By definition of $\Lambda$, we have
\[ \beta-\sigma_\alpha(\beta)\in\bbZ\alpha \]
for all roots $\alpha$, where $\sigma_\alpha$ is the reflection in $\alpha$.
In particular, for all $s'\in S$, we have
\[ \alpha_s-\sigma_{\alpha_{s'}}(\alpha_s)\in m\bbZ\alpha_{s'}, \]
which is to say that the column\footnote{In Bourbaki, row} corresponding to $s$ in the Cartan matrix is divisible by $m$.  This only occurs in type $C_n$, in the case where $\alpha_s$ is the long simple root and $m=2$.~\cite[Tables I-IX, item (XIII)]{Bourbaki}.
\end{proof}

Since there are only two simple\footnote{In the sense of having irreducible root system} groups of Lie type $C_n$ (the simply-connected one and the adjoint one), the preceding three results show:

\begin{thm}
Let $G$ be a simple group of Lie type over $\F_q$ other than a symplectic group (i.e.~the simple simply-connected group of type $C_n$).  Then the action of $T$ on each root subgroup $X_\alpha$ of $G$ induces a surjective map $T\to\F_q^\times$.
\end{thm}

\section{Nilpotence of generators}
Unfortunately, the injection
\[ H^*(GL_2\F_{p^r};\F_p)\to H^*(G;\F_p) \]
does not respect cup products.  However, in partial compensation, it does respect the Steenrod operations.  Importantly, this is enough to determine which elements of the image are nilpotent.

Take $p=2$, and recall that the Steenrod algebra action on any cohomology ring $H^*(G;\F_2)$ satisfies
\[ \Sq^i(x)=x^2\mbox{ when }\deg(x)=i. \]
This implies
\[ \Sq^{2^{k-1}i}\Sq^{2^{k-2}i}\cdots\Sq^i(x)=x^{2^k} \]
where $\deg(x)=i$.  In particular, $x$ is nilpotent if and only if
\[ \Sq^{2^{k-1}i}\Sq^{2^{k-2}i}\cdots\Sq^i(x)=0 \]
for some large enough $k$.
It follows that nilpotence\footnote{Here we are using Lemma~\ref{nilpotentcomponents} to reduce to the homogeneous case.} is preserved by any map between two such cohomology rings respecting the Steenrod operations.  If the map is injective, non-nilpotence is also preserved.  The same is true by a similar argument when $p$ is odd.

So, using the results of section~\ref{gl2nilpotency} (specifically Proposition~\ref{gl2nonnilp}), we get:
\begin{prop}
The injection
\[ \phi:H^*(GL_2\F_{p^r};\F_p)\to H^*(G;\F_p) \]
of Theorem~\ref{injection} preserves nilpotence and nonnilpotence.  In particular:
\begin{enumerate}[(a)]
\item When $p=2$, there are no (nonzero) nilpotent elements in the image.\footnote{This is also clear from the proof of Theorem~\ref{injection}, since these elements are detected on an elementary abelian $2$-subgroup in $G$.}
\item When $p$ is odd, the image $\phi(H^i(GL_2\F_{p^r};\F_p))\subset H^i(G;\F_p)$ contains only nilpotent elements in degrees $r(2p-3)\leq i<r(2p-2)$.
\item There is a nonnilpotent element in $H^{r(2p-2)}(G;\F_p)$.
\end{enumerate}
\end{prop}

\section{Chern class interpretation}\label{chernclass}
In this section, take $r=1$ as well as $n\leq p$.  We give an alternate construction of the nonzero class in $H^{2p-3}(GL_n\F_p;\F_p)$, the case $G=GL_n\F_p$ of Corollary~\ref{nonvanishing}.

In his thesis, Barbu~\cite{Barbu} constructed a nonzero class in $H^{2p-2}(GL_n\F_p;\F_p)$, by building it inductively on the Sylow $p$-subgroup and checking stability.  He conjectured that it is the Bockstein of a class in degree $2p-3$.  We begin by supplying an alternate construction of Barbu's degree $2p-2$ class as a Chern class of a natural permutation representation.

Consider the action of $GL_n\F_p$ on the finite set $X=\F_p^n$.  We claim that its Chern class
\[ c_{p-1}(X)\in H^{2p-2}(GL_n\F_p;\F_p) \]
is nonzero.  We will check it, as Barbu did, by restricting to the
cyclic subgroup of order $p$ generated by
\[ g=\begin{bmatrix} 1&1&&\cdots&0\\&1&1&&\vdots\\&&\ddots&\ddots&\\&&&1&1\\&&&&1\end{bmatrix} \in GL_n\F_p. \]
Since $g$ acts on $X$ with $p$ fixed points and $(p^n-p)/p=p^{n-1}-1$ free orbits, we have
\[ c(X)|_{\langle g\rangle}=(1-u^{p-1})^{p^{n-1}-1}=1+u^{p-1}+\cdots, \]
showing that $c_{p-1}(X)\in H^{2p-2}(GL_n\F_p;\F_p)$ is nonzero, and in fact nonnilpotent.  (Here $\langle u\rangle=H^2(\langle g\rangle;\F_p)$.)

The benefit of this construction (besides its concision) is that it gives the nonzero class in degree $2p-2$ conveniently as the image of the universal Chern class $c_{p-1}$ under the homomorphism\footnote{This does not contradict the stable vanishing of the cohomology of $GL_n\F_p$, because the stabilization maps $GL_n\F_p\to GL_{n+1}\F_p$ are not compatible with those of the symmetric groups.}
\[ H^*(S_\infty;\F_p)\to H^*(GL_n\F_p;\F_p) \]
induced by the action of $GL_n\F_p$ on $X$.  So we may infer its properties from those of $H^*(S_\infty;\F_p)$, which is fairly well understood~\cite{Nakaoka,Nakaoka2,Mann}.
In particular, $c_{p-1}(X)$ must be the Bockstein of a (nonzero) element in $H^{2p-3}(GL_n\F_p;\F_p)$ because $c_{p-1}$ is the Bockstein of an element in $H^{2p-3}(S_\infty;\F_p)$.
To quickly see this last statement, observe that
\[ H^*(S_p;\F_p)=H^*(\F_p;\F_p)^{\F_p^\times} \]
is the subalgebra of
\[ H^*(\F_p;\F_p)=\F_p\langle x\rangle\otimes\F_p[y] \]
(where $y=\beta x$)
generated by $xy^{p-2}$ and
\[ \beta(xy^{p-2})=y^{p-1}=-c_{p-1}. \]
On the other hand, 
inclusion $S_p\to S_\infty$ induces an isomorphism on cohomology below degree $4p-6$ (see Lemma~\ref{Sinfty}).

Now, we have constructed nonzero classes on $GL_n\F_p$ in degrees $2p-2$ and $2p-3$, connected by a Bockstein, explicitly verifying Barbu's conjecture.

\begin{lemma}\label{Sinfty}
Let $p$ be a prime.
Restriction \[ H^i(S_\infty;\F_p)\to H^i(S_p;\F_p) \]
is an isomorphism for $i<4p-6$.
\end{lemma}
\begin{proof}
The inclusions between symmetric groups always induce surjections on cohomology~\cite[Thm.~5.8]{Nakaoka2}, so it suffices to show that $H^i(S_\infty;\F_p)$ is detected on $C_p\leq S_p$ in the desired range.  
Now, the cohomology of symmetric groups is detected on their $p$-tori~\cite[Lem.~2.7]{Mann}, which are products of regular representations of $C_p^k$ (for various $k$).
But the regular representation of $C_p^k$ is normalized by $GL_k\F_p$, so, if $k>1$,
\begin{align*}
\im(H^i(S_\infty;\F_p)\to H^i(C_p^k;\F_p))&\leq H^i(C_p^k;\F_p)^{GL_k\F_p} \\
&\leq H^i(\F_{p^k};\F_p)^{\F_{p^k}^\times},
\end{align*}
the latter vanishing (Proposition~\ref{gl2vanishing}) below degree $k(2p-3)\geq 4p-6$.  Hence we need only consider products of the regular representation of $C_p$.  The lowest nonzero element of $H^*(C_p;\F_p)^{\F_p^\times}$ is in degree $2p-3$, so the cohomology of $S_\infty$ is detected on (conjugates of) a single subgroup $C_p$ below degree $2(2p-3)=4p-6$, as desired.
\end{proof}

\begin{rem}
The construction of this section shows that (for $n\leq p$) the map
\[ BGL_n\F_p^+\to BS_\infty^+ \]
induced by the natural action on $\F_p^n$
is nonzero on $H^{2p-3}(-;\F_p)$.
So it is not merely a coincidence that the lowest mod-$p$ (co)homology of $GL_n\F_p$ coincides with the lowest nontrivial $p$-torsion in the stable homotopy groups of spheres (which are the homotopy groups of $BS_\infty^+$ by the Barratt-Priddy-Quillen Theorem~\cite{BPQ}).

\end{rem}

\begin{rem}
When $r>1$, the technique of this section does not produce a nonzero Chern class of $GL_n\F_{p^r}$ in low degree.
The Chern classes around degree $r(2p-2)$ of the natural permutation representation of $GL_n\F_{p^r}$ are zero when restricted to an elementary abelian subgroup $A$ of rank $r$ whose nontrivial elements are regular unipotent (as in Corollary~\ref{ptorus}).  This is because such an $A$ acts on $\F_{p^r}^n$ with one fixed line, and freely on the remaining vectors.  That is, the natural permutation representation restricts to $A$ as a sum of fixed points and copies of the regular representation.  But the latter has no nonzero Chern classes until degree $2(p^r-p^{r-1})$.

This Chern class construction nevertheless served as the intuitive basis for the transfer construction of Theorem~\ref{injection}.
\end{rem}

Finally, we remark that this construction of the class in degree $2p-3$ indeed agrees with the one given in Corollary~\ref{nonvanishing}, up to a nonzero scalar multiple.  Likely we could check this directly, but that will be unnecessary in light of Theorem~\ref{onedimensional}, which shows that the degree $2p-3$ cohomology is in fact one-dimensional.

\section{Exponent of $U$}
We will show:
\begin{prop}\label{exponent}
In a finite group of Lie type over $\F_{p^r}$ with Coxeter number at most $p$, the Sylow $p$-subgroup $U$ has exponent $p$.
\end{prop}

Let $h\leq p$ be the Coxeter number.  Since $U$ is generated by the positive root subgroups, their height function yields a filtration of $U$, with length $h-1$:
\[ U_{\hgt\geq k} = \langle X_\alpha \mid \alpha\in\Phi^+, \hgt\alpha\geq k\rangle. \]
Since the root subgroups $X_\alpha$ are abelian and satisfy the Chevalley commutator formula
\[ [X_\alpha,X_\beta]\leq\langle X_{i\alpha+j\beta} \mid i,j>0 \rangle \]
for all $\alpha,\beta\in\Phi$ with $\alpha\neq\pm\beta$,
it follows that the above filtration of $U$ is a central series.  In particular, the nilpotence class of $U$ is at most $h-1$.

Using this central series, we show that, when the Coxeter number is at most $p$, every nontrivial element of $U$ has order $p$.  This follows from Hall's ``commutator collection'' trick~\cite[Cor.~12.3.1]{Hall}:

\begin{prop}[Hall]
Let $P$ be a $p$-group of nilpotence class less than $p$.  Then for $a_i\in P$, we have
\[ (a_1\cdots a_r)^p = a_1^pa_2^p\cdots a_r^pS_1^p\cdots S_\ell^p, \]
where $S_i\in[P,P]$.
\end{prop}

\begin{cor}\label{expcor}
Let $P$ have a central series $1=P_0\leq\cdots\leq P_k=P$ of length $k<p$.  Suppose each $P_i$ is generated by elements of order $p$.  Then $P$ has exponent $p$.
\end{cor}

\begin{proof}
We use induction on $k$.  All elements of $P$ can be written as a product $a_1\cdots a_r$ with $a_i^p=1$.  Applying the proposition, we see that it suffices to show that $S^p=1$ for all $S\in[P,P]\leq P_{k-1}$.  This is true because $P_{k-1}$ satisfies the induction hypothesis.
\end{proof}

\begin{proof}[Proof of Proposition~\ref{exponent}]
Each level of the height filtration for $U$ is generated by root subgroups, each of which has exponent $p$.  Also the maximum height is $h-1\leq p-1$.  Hence the conditions of Corollary~\ref{expcor} are satisfied, showing that $U$ has exponent $p$.
\end{proof}

\chapter{On $H^{\lowercase{r(2p-3)}}(GL_{\lowercase{n}}\F_{\lowercase{p^r}};\F_{\lowercase{p}})$}\label{dimensionch}
For this chapter, we focus completely on the finite general linear groups $GL_n\F_{p^r}$.  

Friedlander-Parshall~\cite{FP} showed that $H^i(GL_n\F_{p^r};\F_p)=0$ for all $0<i<r(2p-3)$.
In the previous chapter (Corollary~\ref{nonvanishing}), we produced a nonzero element of cohomology in the lowest potentially nonzero degree $r(2p-3)$, showing
\[ \dim H^{r(2p-3)}(GL_n\F_{p^r};\F_p)\geq1, \]
for $2\leq n\leq p$.  In this chapter, building on ideas used by Quillen~\cite[sec.~11]{QuillenK}, Cline-Parshall-Scott-van der Kallen~\cite[sec.~5]{CPSvdK}, and Friendlander-Parshall~\cite[sec.~7]{FP}, we attempt to provide a reverse inequality, both showing that the above cohomology group is one-dimensional, and also that the condition $n\leq p$ is sharp.  We are able to do this when $r=1$ or when $p=2$:

\begin{thm}\label{lowestGL}
Assume either that $r=1$ or that $p=2$.  Then
\begin{align*}
\dim H^{r(2p-3)}(GL_n\F_{p^r};\F_p)=
\begin{cases}
1, & 2\leq n\leq p, \\
0, & n>p.
\end{cases}
\end{align*}
\end{thm}
\noindent In particular, the classes constructed in section~\ref{cohomology} generate their respective cohomology groups, and there is no other cohomology in this degree.  More precisely, the map of Theorem~\ref{injection} is surjective in degree $r(2p-3)$.

The case where $p$ is odd and $r>1$ remains elusive, because of a lack of knowledge of the cohomology of the ``hook subgroup''
\[ \begin{bmatrix} 1&*&*\\&I_{n-2}&*\\&&1 \end{bmatrix}\leq GL_n\F_{p^r}. \]


\section{Cohomology detection on associated graded groups}
We begin by constructing machinery helpful for studying the lowest (potentially) nontrivial cohomology group (of the general linear groups or their Borel subgroups).  It works by comparing the cohomology of the $p$-Sylow subgroup to the cohomology of its associated graded group (with respect to a central filtration).  This can be done relative to a collection of subgroups, and equivariantly with respect to its normalizer.


\subsection{Spectral sequence background}
Fix a coefficient field $\F$.  To be clear, here is what we mean by the term ``spectral sequence'' for this section:

\begin{defn}
A spectral sequence is a sequence of $\bbN$-graded cochain complexes of $\F$-vector spaces (``pages'') $E_2^*,E_3^*,E_4^*,\dots$, each of which is the homology of the last.  To avoid convergence subtleties we shall further stipulate that the differential on $E_r$ vanishes below degree $r-1$, ensuring that the sequence stabilizes in each degree.  Let $E_\infty^*$ denote the stable value.  We say that the spectral sequence converges to a graded vector space $H^*$ (called ``the abutment'') if there is a filtration $F_1\subset F_2\subset\cdots$ of $H^*$, finite in each degree, such that $E_\infty^*$ is the associated graded space to the filtration.  A map between spectral sequences $E^*_*,{E'}^*_*$ converging to $H^*,{H'}^*$ (resp.) is defined as a map $f:E_2^*\to {E'}^*_2$ together with a map $g:H^*\to {H'}^*$, such that the maps induced by $f$ on each page commute with all of the successive differentials, and the map induced by $f$ on the $E_\infty$ page equals the associated graded map induced by $g$.
\end{defn}

\begin{rem}
Henceforth, we consider the abutment as part of the data of a spectral sequence (and the map on abutments as part of a map of spectral sequences), although these data are not uniquely determined by the pages of the spectral sequence.  Also, this definition does not keep track of filtration degrees, because they play no role in the theorem we are building up to.
\end{rem}

\begin{lemma}\label{filteredvs}
\begin{enumerate}[(a)]
\item Let $A,B$ be vector spaces equipped with finite filtrations, and $f:A\to B$ a filtration-preserving homomorphism.  Then
\[ \dim\ker(A\to B)\leq\dim\ker(\gr A\to\gr B). \]
\item Let $f:E_*^*\to {E'}_*^*$ be a map of spectral sequences with map $H^*\to {H'}^*$ on abutments.  Then for all $d$,
\[ \dim\ker(H^d\to {H'}^d)\leq\dim\ker(E_\infty^d\to {E'}_\infty^d). \]
\end{enumerate}
\end{lemma}

\begin{proof}
Part (b)  follows by applying part (a) to the degree $d$ part (whose filtration is finite).
For part (a), it suffices by induction to consider the case where the filtrations have length two:
\[ 0\leq A'\leq A \+ 0\leq B'\leq B. \]
In this case, the Snake Lemma gives a left exact sequence
\[ 0\to\ker(A'\to B')\to\ker(A\to B)\to\ker(A/A'\to B/B'), \]
which shows
\begin{align*}
\dim\ker(A\to B)&\leq\dim\ker(A'\to B')+\dim\ker(A/A'\to B/B') \\
&=\dim\ker(\gr A\to\gr B).
\qedhere
\end{align*}
\end{proof}

\begin{lemma}\label{sskernel}
Fix $d\in\bbZ$. \begin{enumerate}[(a)]
\item Let $f^\bullet:C^\bullet\to D^\bullet$ be a map of cochain complexes.  
If $D^{d-1}=0$ (or if $\delta:D^{d-1}\to D^d$ is zero), then
\[ \dim\ker(H^d(C^\bullet)\to H^d(D^\bullet))\leq\dim\ker(C^d\to D^d). \]
\item Let $f:E_*^*\to {E'}_*^*$ be a map of spectral sequences with map $H^*\to {H'}^*$ on abutments.  If ${E'}_2^{d-1}=0$ (or $d=1$), then
\[ \dim\ker(H^d\to {H'}^d)\leq\dim\ker(E_2^d\to {E'}_2^d). \]
\end{enumerate}
\end{lemma}

\begin{proof}
Let $K^\bullet=\ker(C^\bullet\to D^\bullet).$  Because the composite
$H^d(K^\bullet)\to H^d(C^\bullet)\to H^d(D^\bullet)$
is zero, we have a map
\[ H^d(K^\bullet)\to\ker(H^d(C^\bullet)\to H^d(D^\bullet)). \]
Using the fact that $D^{d-1}=0$, we see that this map is surjective.
Since $H^d(K^\bullet)$ is a subquotient of $K^d$, this proves part (a).

For part (b), repeatedly apply part (a) to get
\[ \dim\ker(E_\infty^d\to {E'}_\infty^d)\leq\dim\ker(E_2^d\to {E'}_2^d); \]
then apply Lemma~\ref{filteredvs}(b).  At each stage, we have ${E'}_r^{d-1}=0$, or else $d=1$ (there are no differentials $E_r^0\to E_r^1$ for $r\geq2$).
\end{proof}

\begin{rem}
In the important special case (of Lemma~\ref{sskernel}(b)) where the map on $E_2$-pages is injective, the result can be strengthened to the following (which we will not use):
\begin{lemma}
If $E_2^i\to {E'}_2^i$ is surjective for $i=d-1$ and injective for $i=d$, then the same things hold for $H^i\to {H'}^i$.
\end{lemma}
\noindent The same remark applies to the detection theorems of the next section.
\end{rem}

That concludes our brief study of spectral sequence maps; now we switch gears and construct some to apply this lemma to.

\subsection{Spectral sequences for centrally filtered groups}
\begin{prop}\label{sequenceofsss}
Let $G$ be a (nilpotent) finite group with central filtration
$1=G_0\leq G_1\leq\cdots\leq G_n=G$.  Write $\gr G=\prod_{i=1}^n(G_i/G_{i-1})$.  Then there is a sequence of spectral sequences beginning with the $\F$-cohomology of the associated graded group, and ending with the cohomology of $G$.  More precisely, there are spectral sequences ${}_1E,\cdots,{}_nE$ with ${}_1E^*_2=H^*(\gr G)$, such that ${}_kE^*_*$ converges to ${}_{k+1}E^*_2$ for each $k<n$, and ${}_nE^*_*$ converges to $H^*(G)$.  These are all natural with respect to filtration-preserving group homomorphisms.
\end{prop}

\begin{proof}
For each $k$ there is a central extension
\[ G_{k+1}/G_k\to G/G_k\to G/G_{k+1}, \]
and hence a spectral sequence beginning with
\[ E_2^* = H^*(G_{k+1}/G_k)\otimes H^*(G/G_{k+1}), \]
converging to $H^*(G/G_k)$.  Tensoring this with $H^*\gr G_k$ (given trivial differentials) yields the desired spectral sequence
\[ H^*(\gr G_{k+1})\otimes H^*(G/G_{k+1}) \Rightarrow H^*(\gr G_k)\otimes H^*(G/G_k). \qedhere \]
\end{proof}

Now we give the ``prototype version'' of our detection theorem (we shall need an equivariant version below).
Notice that a central filtration of $G$ is inherited by any subgroup $K\leq G$, and we may regard $\gr K$ as a subgroup of $\gr G$.

\begin{thm}[Detection Theorem]\label{detectionthm}
Let $G$ be a group with (chosen) central filtration.  Fix $d>0$.  Suppose that $K_1,\dots, K_k$ are subgroups such that, for all $\ell$,
\[ H^{d-1}(\gr K_\ell)=0 \]
(or $d=1$).  Then
\[ \dim\ker\left(H^d(G)\to\prod_{\ell=1}^k H^d(K_\ell)\right)\leq
\dim\ker\left(H^d(\gr G)\to\prod_{\ell=1}^k H^d(\gr K_\ell)\right). \]
In particular, if $H^d(\gr G)$ is detected on $\{\gr K_\ell\}$ then $H^d(G)$ is detected on $\{K_\ell\}$.
\end{thm}

\begin{proof}
Apply Proposition~\ref{sequenceofsss} (including naturality) to both $G$ and $K_\ell$ to get two sequences of spectral sequences, and compatible restriction maps.  Then take the product of the latter over the various $\ell$ (a finite product of spectral sequences is a spectral sequence, converging to the product of the abutments).  Now we have a sequence of spectral sequence maps, which has at either end the two maps displayed in the theorem.  Repeatedly applying Lemma~\ref{sskernel}(b) finishes the proof.
\end{proof}

Now for the $T$-equivariant version of the proposition and theorem.
A finite group $T$ is called ``nonmodular'' if $\chr\F$ does not divide $|T|$.

\begin{prop}\label{tsequenceofsss}
Let $G$ be a (nilpotent) group with central filtration $1=G_0\leq G_1\leq\cdots\leq G_n=G$.  Suppose that a nonmodular group $T$ acts on $G$ by automorphisms, preserving the filtration. Then there are spectral sequences ${}_1E,\cdots,{}_nE$ with ${}_1E^*_2=H^*(\gr G)^T$, such that ${}_kE^*_*$ converges to ${}_{k+1}E^*_2$ for each $k<n$, and ${}_nE^*_*$ converges to $H^*(G)^T$.  These are all natural with respect to filtration-preserving $T$-equivariant group homomorphisms.
\end{prop}

\begin{proof}
Apply Proposition~\ref{sequenceofsss}, and notice that $T$ acts on each of the spectral sequences, by naturality.  By taking $T$-invariants everywhere, we will be done by the following lemma.
\end{proof}

\begin{lemma}\label{nonmodularss}
Suppose $E_*^*$ is a spectral sequence converging to $H^*$, and a nonmodular group $T$ acts on it by spectral sequence automorphisms.  Then $(E_*^*)^T$ is a spectral sequence converging to $(H^*)^T$.
\end{lemma}

\begin{proof}
To determine that $(E_*^*)^T$ is a spectral sequence, we need to show that the functor $(\cdot)^T=\hom_{\F T}(\F,\cdot)$ commutes with taking homology on cochain complexes of $\F T$-modules, i.e.~is exact.  This is true because, as $\chr\F \nmid\left|T\right|$, $\F$ is projective as an $\F T$-module.
To check that $(E_*^*)^T$ converges to $(H^*)^T$, we need to show that
$\gr (M^T)=(\gr M)^T$ as subspaces of $\gr M$,
for a finitely-filtered $\F T$-module $M$.  In turn, that requires showing that, for $N\leq M$, the map $M^T/N^T\to (M/N)^T$ is an isomorphism.  Again this follows by exactness of $T$-invariants.
\end{proof}

Now we are ready for the main theorem, giving conditions under which detection of $T$-invariant cohomology can be checked on associated graded groups.

\begin{thm}[Equivariant Detection Theorem]\label{tdetectionthm}
Let $G$ be a group with (chosen) central filtration.  Let a nonmodular group $T$ act on $G$, preserving the filtration.  Fix $d>0$.  Suppose that $K_1,\dots, K_k$ are subgroups, preserved by $T$, such that for each $\ell$,
\[ H^{d-1}(\gr K_\ell)^T=0. \]
(or $d=1$).  Then
\begin{align*}
&\dim\ker\left(H^d(G)^T\to\prod_{\ell=1}^k H^d(K_\ell)^T\right)\leq \\
&\dim\ker\left(H^d(\gr G)^T\to\prod_{\ell=1}^k H^d(\gr K_\ell)^T\right).
\end{align*}
In particular, if $H^d(\gr G)^T$ is detected on $\{\gr K_\ell\}$ then $H^d(G)^T$ is detected on $\{K_\ell\}$.
\end{thm}

\begin{proof}
Same as before.
\end{proof}

Incidentally, the special case where there are no $K_i$'s is also useful:

\begin{cor}\label{tvanishing}
Let $G$ be a group with (chosen) central series.  Let a nonmodular group $T$ act on $G$, preserving the filtration.  If $H^i(\gr G)^T=0$ for $0<i<d$, then $H^i(G)^T=0$ for $0<i<d$.
\end{cor}

\section{Restriction to a root subgroup}
Now we return to studying the $\F_p$-cohomology of the finite general linear groups $GL_n\F_{p^r}$.  For the remainder of this chapter, all cohomology is with $\F_p$ coefficients, which will be suppressed from the notation.

The following vanishing result is equivalent to Proposition~\ref{gl2vanishing} from section~\ref{gl2} (using the spectral sequence for the obvious extension, which collapses to the vertical edge):
\begin{lemma}\label{gl2vanishingv2}
$H^i(\F_{p^r}\semidirect\F_{p^r}^\times;\F_p)=0$
for $0<i<r(2p-3)$.
\end{lemma}

This implies a vanishing range for the image of restriction from $GL_n\F_{p^r}$ to its ``right edge subgroup:''

\begin{cor}\label{rightedge}
Restriction
\[ H^i(GL_n\F_{p^r})\to H^i\left(\begin{bmatrix} I_{n-1} & \ast \\ &1 \end{bmatrix}\right) \]
vanishes for
\[ 0<i<r(n-1)(2p-3). \]
In particular, it vanishes in degree $i=r(2p-3)$ when $n>2$.
\end{cor}

\begin{proof}
The restriction factors through
\[ H^i\left(\begin{bmatrix} \F_{p^{r(n-1)}}^\times & \ast \\ &1 \end{bmatrix}\right) \]
whose cohomology vanishes in the given range by Lemma~\ref{gl2vanishingv2}.  Here we are regarding $\F_{p^{r(n-1)}}^\times$ as a subgroup of $GL_{n-1}\F_{p^r}$ via a choice of basis $\F_{p^{r(n-1)}}\iso\F_{p^r}^{n-1}$.
\end{proof}

\begin{rem}
Using Lemma~\ref{edgegroup} in place of Lemma~\ref{gl2vanishingv2}, we could prove a similar vanishing range for restriction to all of the ``block $p$-tori'' of~\cite{MP}:
\begin{cor}
Restriction
\[ H^i(GL_n\F_{p^r})\to H^i\left(\begin{bmatrix} I_k & \ast \\ & I_{n-k} \end{bmatrix}\right) \]
vanishes for
\[ 0<i<r(p-1)\max\{k,n-k\}. \]
In particular, it vanishes in degree $i=r(2p-3)$ when $n>2$.
\end{cor}
\end{rem}

The same vanishing result holds a fortiori for restriction to any root subgroup
\[ E_{\ell m}=\set{A\in GL_n\F_{p^r}}{A_{ij}=\delta_{ij}\mbox{ unless }
i=\ell\mbox{ and }j=m}, \]
since a conjugate of it is contained in the right edge subgroup:

\begin{cor}\label{rootsubgroup}
Restriction
\[ H^i(GL_n\F_{p^r})\to H^i(E_{\ell m}) \]
vanishes for
\[ 0<i<r(n-1)(2p-3). \]
In particular, it vanishes in degree $i=r(2p-3)$ when $n>2$.
\end{cor}

\begin{rem}
At least when $r=1$, we could strengthen the vanishing bound in Corollary~\ref{rightedge} (and hence also Corollary~\ref{rootsubgroup}) to be exponential in $n$ rather than linear, by studying invariants of the full normalizer (``Dickson-Mui invariants'') instead of just invariants by the multiplicative group.
Specifically~\cite{Mui}, the invariants
$H^i(C_p^{n-1};\F_p)^{GL_{n-1}\F_p}$ vanish for
\[ 0<i<n-1+2(p-2)(1+p+\cdots+p^{n-2}). \]
(For $n>2$, this is higher than $(n-1)(2p-3)$.)
\end{rem}

\section{Torus-stable filtration}\label{filtration}
Recall $U_n=U_n\F_{p^r}$ is the group of unipotent upper triangular matrices in $GL_n\F_{p^r}$, and its Sylow $p$-subgroup.  Being a $p$-group, $U_n$ is of course nilpotent.  Indeed, there is a particularly nice central filtration of $U_n$ by ``superdiagonals'' which is preserved by the action of the torus $T$ (of diagonal matrices).  The idea is use the preceding theory to compare the cohomology of $U_n$ to the cohomology of the associated graded group $\gr U_n$, which is an elementary abelian $p$-group (more specifically, an $\F_{p^r}$-vector space).

For each $0\leq i\leq n-1$, define
\begin{align*}
G_k &= \set{A\in U_n}{A_{ij}=0\mbox{ when }0<j-i<n-k}\\
&= \langle E_{ij}\mid n-k\leq j-i<n\rangle.
\end{align*}
This filtration of $U_n$ is central (in fact it is the lower central series of $U_n$), and is preserved by $T$ (since the $E_{ij}$ are).  Its associated graded group
\[ \gr U_n = \prod_{i<j} E_{ij} \]
is elementary abelian.  Hence its cohomology is a polynomial algebra when $p=2$ and a polynomial tensor an exterior algebra when $p$ is odd.
In fact, $\gr U_n$ is an $\F_{p^r}$-vector space, on which $T$ acts $\F_{p^r}$-linearly.  The $E_{ij}$ are $T$-stable subspaces with $T$-weight $(t_k)\mapsto t_i/t_j$.

\section{The case $p=2$}
For this section, let $p=2$ (and take $\F_2$-cohomology).  Here we have $r(2p-3)=r$.  Recall (Proposition~\ref{gl2vanishing}) that $\dim H^r(GL_2\F_{2^r})=1$, so we just need to show $H^r(GL_n\F_{2^r})=0$ for $n>2$.  Along the way, we'll compute the lowest cohomology group $H^r(B_n)$ of the Borel subgroup
\[ B_n = \set{A\in GL_n\F_{2^r}}{A_{ij}=0\mbox{ when }i>j} \]
for all $n$.

Let $T=(\F_{2^r}^\times)^n$ be the diagonal torus, so that $B_n=U_n\rtimes T$,
and we have
\[ H^*(B_n)=H^*(U_n)^T. \]
We begin by calculating $H^r(\gr U_n)^T$, the $T$-invariant cohomology of the associated graded group $\gr U_n=\prod_{i<j}E_{ij}$ with respect to the filtration described in section~\ref{filtration}.

\begin{lemma}\label{grrootdetection}
\begin{enumerate}[(a)]
\item For all $0<d<r$,
\[ H^d(\gr U_n)^T = 0; \]
\item For all $0<d<r$ and $i<j$,
\[ H^d(E_{ij})^T = 0; \]
\item The map
\[ \bigoplus_{i<j}H^r(E_{ij})^T\to H^r(\gr U_n)^T \]
induced by the projections $\gr G\to E_{ij}$ onto the factors is surjective;
\item The restriction
\[ H^r(\gr U_n)^T\to\prod_{i<j}H^r(E_{ij})^T \]
is an isomorphism.
\end{enumerate}
\end{lemma}

\begin{proof}
Note that each $E_{ij}\leq U_n$ is canonically isomorphic to $\F_{2^r}$, and $T$ acts on it with weight
\[ (t_1,\dots,t_n)\mapsto t_i/t_j. \]
For this proof, we take cohomology with coefficients extended to $\F_{2^r}$.  This does not affect any of the statements,\footnote{Because $T$-invariants commute with extension of the coefficient field}
and has the benefit of providing a canonical (up to scalar multiple) generating set of $T$-eigenvectors, by Proposition~\ref{vectorspacecoh}:
\[ H^*(\gr U_n) = \F_{2^r}[z_{ijk}\mid i<j,\ 0\leq k<r] \]
such that $T$ acts on $z_{ijk}$ with weight 
\[ (t_1,\dots,t_n)\mapsto (t_i/t_j)^{2^k}. \]
The cohomology therefore has a basis consisting of the monomials in the $z_{ijk}$'s, which are eigenvectors for the action of $T$.  
Hence the $T$-invariants are spanned by the $T$-invariant monomials. 

We claim that there are no such invariant monomials below degree $r$, and that all those in degree $r$ are pulled back from one of the $E_{ij}$ (which is to say that they contain only factors $z_{ijk}$ with one fixed $(i,j)$).
We prove the claim by induction on $n$.  Let us suppose that $\prod_{i<j}\prod_{k}z_{ijk}^{a_{ijk}}$ is a nontrivial $T$-invariant monomial for some $a_{ijk}\geq0$.
Let $b_{ij}=\sum_ka_{ijk}$.  We may assume without loss of generality that
\[ \sum_{j>1}b_{1j}, \sum_{i<n}b_{in}>0. \]
(Otherwise, our monomial is pulled back from a copy of $\gr U_{n-1}$, and we may apply the inductive hypothesis).
By considering the action of $\F_{2^r}^\times \times1\times\cdots\times1\subset T$, we know that
\[ (2^r-1)\mid\sum_{1<j}\sum_{k=0}^{r-1} 2^ka_{1jk}\neq0, \]
whence it follows (using Lemma~\ref{Quillenlemma}) that 
\[ \sum_{1<j} b_{1j} \geq r. \]
Similarly,
\[ \sum_{i<n} b_{in} \geq r. \]
Hence if our invariant monomial has degree at most $r$, then we must have
$b_{1n}=r$, and $b_{ij}=0$ for all $(i,j)\neq(1,n)$.  In particular, its degree is precisely $r$, and it is in the image of pullback from $E_{1n}$, as desired.  This proves the claim, and establishes parts (a) and (c).  Part (b) follows from (a) since each $E_{ij}$ is a $T$-equivariant retract of $\gr U_n$.   To get part (d), notice that the composition
\[ \bigoplus_{i<j}H^r(E_{ij})^T\to H^r(\gr U_n)^T\to
    \prod_{i<j}H^r(E_{ij})^T \]
of the maps in (c) and (d) is the identity, hence both maps are isomorphisms.
\end{proof}

\begin{rem}\label{lowvanishing2}
With Corollary~\ref{tvanishing}, Lemma~\ref{grrootdetection}(a) shows that
for $0<d<r$,
\[ H^d(U_n)^T=0. \]
Of course it follows that $H^d(GL_n\F_{2^r})=H^d(B_n)=0$ in those degrees, since $U_n$ is a 2-Sylow subgroup.
This is (essentially) the argument used by Friedlander-Parshall~\cite{FP}.
\end{rem}

Now we have all the ingredients to compute $H^r(GL_n\F_{2^r})$.
\begin{thm}\label{lowestcoh2}
We have
\[ \dim H^r(GL_n\F_{2^r};\F_2)=
\begin{cases} 1&\mbox{if }n=2,\\0&\mbox{if }n>2. \end{cases} \]
\end{thm}

\begin{proof}
For the case $n=2$, see Proposition~\ref{gl2vanishing}.
Lemma~\ref{grrootdetection}(b,d) shows that the hypotheses of Theorem~\ref{tdetectionthm} are satisfied, establishing that
\[ H^r(U_n)^T\to\prod_{i<j}H^r(E_{ij})^T \]
is injective.
It follows that restriction
\[ H^r(GL_n\F_{2^r})\to\prod_{i<j}H^r(E_{ij}) \]
is injective.
However, when $n>2$, restriction to each root subgroup vanishes by Corollary~\ref{rootsubgroup}, and this map is zero; consequently, $H^r(GL_n\F_{2^r})=0$.
\end{proof}

We have now completed our main objective in the case $p=2$; however, with just a bit more work, we can also calculate the lowest nonzero
cohomology group $H^r(B_n)$ of the Borel subgroup.  We will express the degree $r$ cohomology of $B_n$ in terms of the cohomology of $B_2$, which embeds into $B_n$ as a retract in several ways.

Recall that $H^d(B_n\F_{2^r};\F_2)=0$ for $0<d<r$ (\cite{FP}; also see Remark ~\ref{lowvanishing2}).  In the lowest nonvanishing degree $r$, we have:
\begin{thm}\label{Bnlowest}
The restriction
\[ H^r(B_n\F_{2^r};\F_2)\to\prod_{k=1}^{n-1} H^r(B_2\F_{2^r};\F_2) \]
is an isomorphism, where the map to the $k$th factor is restriction to the block diagonal subgroup
\[ \mat{3}{I_{k-1}&0&0 \\ &B_2&0 \\ &&I_{n-k-1}}. \]
In particular,
\[ \dim H^r(B_n\F_{2^r};\F_2)=n-1. \]
\end{thm}

\begin{proof}
Since $U_n$ is the Sylow 2-subgroup of $B_n=U_n\semidirect T$, it suffices to show that
the restriction
\[ H^r(U_n)^T\to\prod_{k=1}^{n-1}H^r(E_{k,k+1})^T \]
is an isomorphism.
As in the proof of Theorem~\ref{lowestcoh2}, we have that
\[ H^r(U_n)^T\to\prod_{i<j}H^r(E_{ij})^T \]
is injective.
To complete the description, we need to refine this information by showing that the image of the map is precisely
$\prod_{j=i+1} H^d(E_{ij})^T$.  That will follow from two facts:
\begin{enumerate}[(i)]
\item Restriction $\displaystyle H^r(U_n)^T\to\prod_{k=1}^{n-1}H^r(E_{k,k+1})^T$ is surjective;
\item Restriction $\displaystyle H^r(U_n)^T\to H^r(E_{ij})^T$ is trivial for $j>i+1$.
\end{enumerate}

The first statement holds because 
the $E_{k,k+1}$ are $T$-equivariant retracts of $U_n$, and intersect trivially, so that
the composition
\[ \bigoplus_{k=1}^{n-1}H^r(E_{k,k+1})^T\to 
H^r(U_n)^T\to\prod_{k=1}^{n-1}H^r(E_{k,k+1})^T \]
equals the identity, and the second map must be surjective.\footnote{Warning: we are only claiming that each individual $E_{k,k+1}\leq U_n$ is a retract.  Their product is not even a subgroup, as they do not normalize one another.}

For statement (ii), by factoring the restriction through an appropriately chosen embedding of $U_3$ in $U_n$, we reduce to the case $n=3$ and $(i,j)=(1,3)$.  This case can be handled using the following Lemma~\ref{centercommutator}.
It applies to $E_{13}\leq U_3$ (since $E_{13}$ is both the center and commutator), and shows that the image
$\im(H^r(U_3)\to H^r(E_{13}))$
is contained in the subalgebra of perfect squares in $H^*(E_{13})$.
Because $H^r(\F_{2^r};\F_2)^{\F_{2^r}^\times}$ trivially intersects the subalgebra of perfect squares (Lemma~\ref{noperfectsquares}), this means that
\[ H^r(E_{13})^T\cap\im\left(H^*(U_3)\to H^*(E_{13})\right)=0 \]
as desired.
\end{proof}

The following lemma regarding extensions by an elementary abelian $p$-group contained in the center and commutator subgroup may have independent interest.
\begin{lemma}\label{centercommutator}
Let $Z\xrightarrow{i} G\xrightarrow{\pi} B$ be an extension of finite groups with
$Z$ an elementary $p$-group
and $Z\leq Z(G)\cap[G,G]$.  Then
\[ \im(i^*)\subset P, \]
where $P\subset H^*(Z;\F_p)$ is the subalgebra generated by Bocksteins $\beta H^1(Z;\F_p)$.  (When $p=2$, this coincides with the set of perfect squares.)
\end{lemma}

\begin{proof}
We take coefficients in $\F_p$, and omit them from the notation.
Because $Z\leq[G,G]$, we know $\pi^*:H^1(B)\to H^1(G)$ is surjective.  It follows that $i^*:H^1(G)\to H^1(Z)$ is zero.  Hence the differential $d_2:H^1(Z)\to H^2(B)$ is injective.  Choose a basis $x_1,\dots,x_n$ for $H^1(Z)$, and let $y_i=\beta x_i$.  Then if $p$ is odd,
\[ H^*(Z) = \F_p\langle x_1,\dots,x_n\rangle\otimes\F_p[y_1,\dots,y_n], \]
while if $p=2$, $y_i=x_i^2$ and
\[ H^*(Z) = \F_2[x_1,\dots,x_n]. \]
We know that $z_i=d_2(x_i)$ are linearly independent in $H^2(B)$.  By the Kudo transgression theorem we know that the $y_i$ are transgressive, so 
\[ P=\F_p[y_1,\dots,y_n]\subset\ker(d_2)\cap E_2^{0*}. \]
We would like to show equality.  First, notice that $H^*(Z)$ is free over $P$ on the $2^n$ generators
\[ \set{X_S=\prod_{i\in S}x_i}{S\subset[n]}. \]
We calculate
\[ d_2(X_S) = \sum_{i\in S} \pm X_{S-i}\otimes z_i\in H^*(Z)\otimes H^2(B). \]
We would like to show that, with the exception of $d_2(X_\emptyset)=0$, these are linearly independent over $P$.  This is the case because the $X_T\otimes z_i$ are all linearly independent over $P$ (since the $X_T$ are independent over $P$ and the $z_i$ are independent over $\F_p$), and each such term occurs in only one such differential.
We can now conclude that
\[ \im(i^*)=E_\infty^{0*}\subset E_3^{0*}=P. \qedhere \]
\end{proof}

\section{Other Lie types when $p=2$}\label{char2lietype}
Before moving on to discuss $GL_n\F_{p^r}$ with $p$ odd, we digress to show that, when $p=2$, the Friedlander-Parshall vanishing range approach can be readily extended to give a vanishing range in other Lie types as well.
Compare with the results of~\cite[sec.~7]{BNP3} and~\cite{Hiller}.

\begin{prop}\label{adjointchar2}
Let $G$ be a finite group of Lie type over $\F_{2^r}$.  Assume that $G$ has adjoint type and that its root system has no component of type $E_8$, $F_4$, or $G_2$.  Then
\[ H^i(G;\F_2)=0\mbox{ for } 0<i<r. \]
\end{prop}

\begin{proof}
Using Corollary~\ref{tvanishing} just as before, it suffices to prove that
\[ H^i(\gr U;\F_{2^r})^T=0\mbox{ for } 0<i<r, \]
where $U$ is a Sylow 2-subgroup, filtered by root height, and $T$ the $\F_{2^r}$-split maximal torus that normalizes it.
We resume all the notation from section~\ref{notationlietype}.

Let
$\chi_*=\hom(\bar\F_2^\times,\bar T)$, the cocharacter group of the ambient algebraic torus, a free abelian group.  Because $\bar T$ is maximally $\F_{2^r}$-split,
every cocharacter $\bar\F_2^\times\to\bar T$ restricts to a homomorphism $\F_{2^r}^\times\to T$.
The adjoint type assumption means that the homomorphism
\[ \chi_*\to\Lambda^\vee=(\bbZ\Phi)^* \]
from the cocharacter lattice into the coweight lattice is surjective.
This means that, for each $s\in S$, there is a map
$x_s:\bar\F_2^\times\to\bar T$ with the property that
\[ \alpha_{s'}\circ x_s=\begin{cases} \id_{\bar\F_2^\times} &\mbox{if }s=s',\\
1 &\mbox{if }s\neq s'. \end{cases} \]
In particular (choosing a generator for $\F_{2^r}^\times\leq\bar\F_2^\times$), there is an element $t_s\in T$ such that $\alpha_s(t_s)$ has order $2^r-1$, and $\alpha_{s'}(t_s)=1$ for all $s'\neq s$.

Now, for any $J\subset S$, define
\[ \Phi_J=\Phi\cap\langle\alpha_s\mid s\in J\rangle, \]
a parabolic sub-root system of $\Phi$.
We'll show by induction on the size of $J$ that
\[ H^i(V_J;\F_{2^r})^T=0 \] for $0<i<r$, where
\[ V_J=\bigoplus_{\alpha\in\Phi_J^+} U_\alpha \]
($U_\alpha$ being one-dimensional over $\F_{2^r}$ with $T$-weight $\alpha$).
Since $V_S=\gr U$, this will suffice.

As before (using Proposition~\ref{vectorspacecoh}), we have
\[ H^*(V_J;\F_{2^r})=\F_{2^r}[z_{\alpha k}\mid\alpha\in\Phi_J^+, 0\leq k<r], \]
where $T$ acts on $z_{\alpha k}$ with weight $2^k\alpha$.
The monomials
\[ z=\prod_{\alpha,k}z_{\alpha k}^{a_{\alpha k}} \]
form a $T$-eigenbasis for $H^*(V_J;\F_2)$.  So our task is to prove that any nontrivial $T$-invariant monomial must have degree
\[\deg(z)=\sum_{\alpha,k}a_{\alpha k} \]
at least $r$.

Accordingly, suppose that $z$ is such a nontrivial $T$-invariant monomial.  For each $s\in J$, let
\[ b_s=\sum_{\alpha\in\Phi_J^+}\sum_{k=0}^{r-1}\alpha_s^*(\alpha)a_{\alpha k}, \]
a nonnegative integer.  Here $\alpha_s^*\in(\bbZ\Phi)^*$ are the fundamental coweights, i.e.~the dual basis to $(\alpha_s)_{s\in S}$.

Suppose first that not all $b_s$ are positive.  Then set
\begin{align*} J'&=\set{s\in J}{b_s>0}\\
&=\set{s\in J}{a_{\alpha k}>0\mbox{ for some }\alpha,k\mbox{ with }\alpha_s\leq\alpha}. \end{align*}
Observe that $a_{\alpha k}=0$ for $\alpha\notin\Phi_{J'}$.  That is, $z$ is in fact a $T$-invariant monomial on $V_{J'}$.  Applying the induction hypothesis, we conclude that $\deg(z)\geq r$ as desired.

Now we treat the case where all $b_s>0$.  Choose any irreducible component of $\Phi_J$ and let $\alpha_0$ be its highest root.  Since the class of Dynkin diagrams with no component of type $E_8$, $F_4$, or $G_2$ is closed under taking subdiagrams, our component is not of the three excluded types.  So, there exists~\cite[Tables I-IX, item (IV)]{Bourbaki} an $s\in J$ such that
\[ \alpha_s^*(\alpha_0)=1. \]
Consequently, $\alpha_s^*(\alpha)\in\{0,1\}$ for all $\alpha\in\Phi_J^+$.

Consider the action of the corresponding element $t_s\in T$ constructed above.
Since $t_s$ fixes our monomial
\[ z= \prod_{\alpha\in\Phi_J^+}\prod_{k=0}^{r-1}z_{\alpha k}^{a_{\alpha k}}, \]
we have
\begin{align*}
1&=\prod_{\alpha,k}\alpha(t_s)^{2^ka_{\alpha k}} \\
&=\prod_{\alpha,k}\prod_{s'\in J}\alpha_{s'}(t_s)^{2^k\alpha_{s'}^*(\alpha)a_{\alpha k}} \\
&=\alpha_s(t_s)^{\sum_{\alpha,k}2^k\alpha_s^*(\alpha)a_{\alpha k}}.
\end{align*}
Because $\alpha_s(t_s)$ has order $2^r-1$, this shows
\[ (2^r-1)\mid\sum_{\alpha\in\Phi_J^+}\sum_{k=0}^{r-1}2^k\alpha_s^*(\alpha)a_{\alpha k}, \]
which is nonzero because $b_s>0$.
From this it follows (using Lemma~\ref{Quillenlemma}) that
\[ r\leq\sum_{\alpha,k}\alpha_s^*(\alpha)a_{\alpha k}\leq\sum_{\alpha,k}a_{\alpha k}=\deg(z), \]
as desired.
\end{proof}

Now, consider what happens when we remove the adjoint type assumption on $G$.
Then the inclusion
\[ \chi_*\to\Lambda^\vee=(\bbZ\Phi)^* \]
is no longer surjective.  Its cokernel is (by definition) the ``cofundamental group'' of $G$, i.e.~the quotient of the fundamental group of the root system $\Phi$ by the fundamental group of $G$: a finite abelian group.
We can no longer realize the fundamental coweights $\alpha_s^*$ by cocharacters; however, we do have that their multiples $e\alpha_s^*$ are realized by cocharacters, where $e$ is the exponent of the cofundamental group.
This means that, for each $s\in S$, there is a map
$x_s:\bar\F_2^\times\to\bar T$ with the property that
\[ \alpha_{s'}\circ x_s=\begin{cases} c\mapsto c^e &\mbox{if }s=s',\\
1 &\mbox{if }s\neq s'. \end{cases} \]
In particular there is an element $t_s\in T$ such that $\alpha_s(t_s)$ has order
\[ \frac{2^r-1}{\gcd(e,2^r-1)}, \]
and $\alpha_{s'}(t_s)=1$ for all $s'\neq s$.
The rest of the argument proceeds similarly, showing:
\begin{thm}
Let $G$ be a finite group of Lie type over $\F_{2^r}$, whose root system has no component of type $E_8$, $F_4$, or $G_2$.  Then
\[ H^i(G;\F_2)=0\mbox{ for } 0<i<r/\gcd(e,2^r-1). \]
Here $e$ is the exponent of the cofundamental group of $G$, i.e.~the quotient of the coweight lattice by the cocharacter lattice.
\end{thm}

In some cases (besides the adjoint groups), the $\gcd$ in the theorem is always one:
\begin{cor}
Let $G$ be a finite group of Lie type over $\F_{2^r}$, whose root system has only components of types $B$, $C$, $D$, and $E_7$.  Then
\[ H^i(G;\F_2)=0\mbox{ for } 0<i<r. \]
\end{cor}
\begin{proof}
In these cases, the exponent of the fundamental group of the root system is a power of 2~\cite[Tables I-IX, item (VIII)]{Bourbaki}.
\end{proof}

Notice that, in this section, we have only ever employed a single coweight $\alpha_s^*$ at a time to establish our bounds, whereas the Friedlander-Parshall argument for $GL_n$ uses two at once.  The reason for this is that, except in type $A$, there are never two distinct $s,s'\in S$ with $\alpha_s^*(\alpha_0)=\alpha_{s'}^*(\alpha_0)=1$.  This is the hindrance to employing the above argument in odd characteristic.  Much more delicate root system combinatorics appear to be necessary there to get a reasonably effective bound.  Were we to make the same argument in odd characteristic $p$, the analog of Proposition~\ref{adjointchar2} would give a bound of $r(p-1)$ for the adjoint groups: lower than the sharp bound in the known cases ($r(2p-3)$, in the simply-laced types for $p$ large enough~\cite{BNP2}).
In contrast, for the adjoint groups when $p=2$, it is not possible to improve these results by finer control of the root system combinatorics; this is because $H^*(\gr U)^T$ contains copies of $H^*(\F_{2^r})^T$, which are nonzero in degree $r$.  It is likely that improvement is possible in the non-adjoint cases however.

\section{The case of odd $p$}
Let $p$ be an odd prime; now we study $H^{r(2p-3)}(GL_n\F_{p^r};\F_p)$ in this case.
The situation with $p$ odd is considerably more complicated due to the presence of exterior generators in the cohomology ring of an elementary abelian $p$-group.
It will no longer be the case that our lowest cohomology is detected on the individual root subgroups $E_{ij}$; now we must consider the following ``hook subgroups.''
For each pair $\ell<m$, define the hook subgroup
\begin{align*}
K_{\ell m} &= \set{x\in U_n}{ x_{ij}=\delta_{ij} \mbox{ unless } \ell=i<j\leq m \mbox{ or } \ell\leq i<j=m } \\
&=\begin{bmatrix}
I_{\ell-1}&0&0&0&0 \\
 &1&*&*&0 \\
 &&I_{m-\ell-1}&*&0 \\
 &&&1&0 \\
 &&&&I_{n-m}
\end{bmatrix}.
\end{align*}
Note that the diagonal torus $T=T_n(\F_{p^r})=(\F_{p^r}^\times)^n$ normalizes each $K_{\ell m}$.  Our interest in these hook subgroups is that they detect $H^{r(2p-3)}(\gr U_n)^T$ and therefore (as will follow) $H^{r(2p-3)}(GL_n\F_{p^r})$; see part (c) of the next lemma.

As before, equip $U_n$ (and therefore each $K_{\ell m}$) with the ``superdiagonal'' central filtration (described in section~\ref{filtration}), so that
\[ \gr U_n = \prod_{i<j} E_{ij} \]
and
\[ \gr K_{\ell m} = \prod_{\substack{\ell=i<j\leq m \mbox{ or } \\ \ell\leq i<j=m}} E_{ij}. \]

As before, we begin by stating the desired detection property on the level of associated graded groups.

\begin{lemma}\label{grhook}
\begin{enumerate}[(a)]
\item For $0<i<r(2p-3)$,
\[ H^i(\gr U_n)^T = 0; \]
\item For $0<i<r(2p-3)$, and $\ell<m$,
\[ H^i(\gr K_{\ell m})^T = 0; \]
\item The restriction map
\[ H^{r(2p-3)}(\gr U_n)^T \to
\prod_{\ell<m} H^{r(2p-3)}(\gr K_{\ell m})^T \]
is injective.
\end{enumerate}
\end{lemma}
\noindent This result is similar to but more involved than Lemma~\ref{grrootdetection}.  We place the proof in its own section after this one.

Part (a) of the lemma shows that we may apply Corollary~\ref{tvanishing} to conclude:
\begin{cor}
For $0<d<r(2p-3)$, $H^d(U_n)^T=0$.
\end{cor}
\noindent Of course it follows that $H^i(GL_n\F_{p^r})=0$ in those degrees; this is (essentially) the argument used by Friedlander-Parshall~\cite{FP}.  Now we will investigate the situation in degree $r(2p-3)$.

Part (b) of the lemma shows that the hypotheses of Theorem~\ref{tdetectionthm} are satisfied, with $G=U_n(\F_{p^r})$, $\{K_k\}=\{K_{ij}\}$, and $d=r(2p-3)$.  Together with part (c) of the lemma, it shows that
\[ H^{r(2p-3)}(U_n)^T\to\prod_{i<j} H^{r(2p-3)}(K_{ij})^T \]
is injective.
Therefore:
\begin{lemma}\label{hookdetection}
Restriction
\[ H^{r(2p-3)}(GL_n\F_{p^r})\to\prod_{i<j} H^{r(2p-3)}(K_{ij}) \]
is injective.
\end{lemma}

This lemma says that $H^{r(2p-3)}(GL_n\F_{p^r})$ is detected on the hook subgroups $K_{ij}$.  However, actually the restriction to every such hook subgroup---except for the largest one---is trivial:

\begin{lemma}\label{smallhook}
The restriction
\[ H^{r(2p-3)}(GL_n\F_{p^r})\to H^{r(2p-3)}(K_{ij}) \]
is zero, unless $(i,j)=(1,n)$.
\end{lemma}

\begin{proof}
Suppose $(i,j)\neq(1,n)$.  By symmetry\footnote{Specifically, the group automorphism obtained by composing the inverse transpose and conjugation by the permutation $[n\cdots 1]$}
we may assume that $i>1$.  By induction on $n$, it suffices to treat the case $(i,j)=(2,n)$: for otherwise, the restriction factors through a smaller general linear group.
Now,
\[ K_{2,n}=\begin{bmatrix}
1&0&0&0 \\
 &1&*&* \\
 & &I_{n-3}&* \\
 & & &1
\end{bmatrix}
\leq GL_n\F_{p^r} \]
is contained in the subgroup

\begin{align*}
L=
\begin{bmatrix}
\multicolumn{2}{c}{\multirow{2}{*}{ $\F_{p^{2r}}^\times$ }}
 &*&& * \\
&&*&& * \\
&&I_{n-3}&&* \\
&&&&1
\end{bmatrix},
\end{align*}
where we view $\F_{p^{2r}}^\times$ as a subgroup of $GL_2\F_{p^r}$ via a choice of basis $\F_{p^{2r}}\iso\F_{p^r}^2$.
Hence it suffices to show the restriction to $L$ vanishes in degree $r(2p-3)$.  Now, there is a split extension
\[ 1\to \F_{p^{2r}}^{n-2}\semidirect\F_{p^{2r}}^\times\to L\xrightarrow{\pi} B\to1, \]
with
\[ B=\begin{bmatrix} I_2&0&0 \\ &I_{n-3}&\ast \\ &&1 \end{bmatrix}. \]
But by Lemma~\ref{edgegroup}, $H^i(\F_{p^{2r}}^{n-2}\semidirect\F_{p^{2r}}^\times)=0$ for $0<i<2r(p-1)$, in particular up to degree $r(2p-3)$.  Hence the spectral sequence for the extension shows that $\pi^*:H^{r(2p-3)}(B)\to H^{r(2p-3)}(L)$ is an isomorphism; equivalently, restriction
\[ H^{r(2p-3)}(L)\to H^{r(2p-3)}(B) \]
is an isomorphism.  Therefore, to see that restriction
\[ H^{r(2p-3)}(GL_n\F_{p^r})\to H^{r(2p-3)}(L) \]
vanishes, it suffices to check that
\[ H^{r(2p-3)}(GL_n\F_{p^r})\to H^{r(2p-3)}(B) \]
vanishes.  But this follows from Corollary~\ref{rightedge}, since $B$ is contained in the right edge subgroup.
\end{proof}

Incidentally, by applying Lemma~\ref{hookdetection} to $GL_{n-1}\F_{p^r}$ and Lemma~\ref{smallhook} to $GL_n\F_{p^r}$, we now have:
\begin{cor}
The restriction
\[H^{r(2p-3)}(GL_n\F_{p^r})\to H^{r(2p-3)}(GL_{n-1}\F_{p^r}) \]
is zero (for all $n$).
\end{cor}

Now, let $K=K_{1n}$ be the largest hook subgroup.  Combined, Lemmas~\ref{hookdetection}~and~\ref{smallhook} show:
\begin{thm}\label{bighook}
Restriction
\[ H^{r(2p-3)}(GL_n\F_{p^r})\to H^{r(2p-3)}(K) \]
is injective.
\end{thm}

Now we have made progress, because the hook group $K$ (also called the ``Heisenberg group'') is a much less complicated group; it is nilpotent of class 2 (an extraspecial $p$-group when $r=1$).  However it is still not easy to calculate its cohomology directly.  When $r=1$, we can fortunately avoid doing so by noticing that almost all of its torus-invariant cohomology is detected on smaller hook subgroups, and therefore not in the image of restriction.

\section{Hook subgroup when $r=1$}
For this section, assume $p$ is odd and $r=1$.
For each $1<i<n$, define
\[ L_i=\set{A\in K}{A_{1i}=A_{in}=0}=\begin{bmatrix}
1&*&0&*&* \\ 
&I_{i-2}&&&* \\
&&1&&0 \\
&&&I_{n-i-1}&* \\
&&&&1
\end{bmatrix}
\leq K. \]

Observe that each $L_i$ is normalized by the diagonal subgroup $T=(\F_p^\times)^n$, and also that each $L_i$ is conjugate in $GL_n\F_p$ to the smaller hook subgroup
\[ K_{2,n}=\begin{bmatrix}
1&0&0&0 \\
 &1&*&* \\
 & &I_{n-3}&* \\
 & & &1
\end{bmatrix}. \]
Consequently, by Lemma~\ref{smallhook}, restriction $H^{2p-3}(GL_n\F_p)\to H^{2p-3}(L_i)$ vanishes.

Therefore, the image of restriction from $GL_n\F_p$ to $K$ is contained in the kernel of restriction to each $L_i$.  We'll show that this kernel is at most one-dimensional.  As above, the statement is given first on the associated graded level (and proven in the next section):

\begin{lemma}\label{gressentialhook}
\begin{enumerate}[(a)]
\item When $0<i<2p-3$,
\[ H^{2p-3}(\gr L_k)^T=0 \]
for all $1<k<n$;
\item We have
\[ \dim\ker\left( H^{2p-3}(\gr K)^T\to\prod_{1<k<n}H^{2p-3}(\gr L_k)^T \right)=
\begin{cases} 1&\mbox{for }2\leq n\leq p, \\ 0&\mbox{for }n>p. \end{cases} \]
\end{enumerate}
\end{lemma}
\noindent
We remark that the analog of part (b) over $\F_{p^r}$ for $r>1$ is false.  This is why the argument of this section does not extend to the case $r>1$.

Using Lemma~\ref{gressentialhook}, we get:
\begin{lemma}\label{essentialhook}
\[ \dim\ker\left( H^{2p-3}(K)^T\to\prod_{1<i<n}H^{2p-3}(L_i)^T \right)\leq
\begin{cases} 1&\mbox{for }2\leq n\leq p, \\ 0&\mbox{for }n>p. \end{cases} \]
\end{lemma}

\begin{proof}
Lemma~\ref{gressentialhook} shows that we may apply Theorem~\ref{tdetectionthm}, with $G=K$, $\{K_k\}=\{L_i\}$, and $d=2p-3$. 
\end{proof}

Now we can easily prove:

\begin{thm}\label{onedimensional}
\[ \dim H^{2p-3}(GL_n\F_p;\F_p) =
\begin{cases} 1&\mbox{for }2\leq n\leq p, \\ 0&\mbox{for }n>p. \end{cases} \]
\end{thm}

\begin{proof}
By Theorem~\ref{bighook},
\[ H^{2p-3}(GL_n\F_p)\to H^{2p-3}(K)^T \]
is injective.  Combined with Lemma~\ref{essentialhook}, this shows that the kernel of
\[ H^{2p-3}(GL_n\F_p)\to\prod_{1<i<n}H^{2p-3}(L_i) \]
has dimension at most one, or zero when $n>p$.  However, as remarked above, this map is actually zero.  Consequently,
\[ \dim H^{2p-3}(GL_n\F_p) \leq
\begin{cases} 1&\mbox{for }2\leq n\leq p, \\ 0&\mbox{for }n>p. \end{cases} \]
The previous chapter---specifically Theorem~\ref{glnnonvanishing}---shows that we have equality.
\end{proof}

\section{$T$-invariants in $H^*(\gr U_n)$}
First we need a utility which helps with proving detection on the associated graded level.

\begin{lemma}\label{kercoker}
Let $V$ be a vector space with basis $S$, and fix $\Lambda\subset\powerset(S)$.
Then
\[ \ker\left(V\to\prod_{A\in\Lambda}\Span A\right)\iso
\coker\left(\bigoplus_{A\in\Lambda}\Span A\to V\right), \]
where $\Span A\to V$ is the inclusion, and $V\to\Span A$ the orthogonal projection with respect to the basis $S$.
\end{lemma}

\begin{proof}
Let $I=\im\left(\bigoplus_{A\in\Lambda}\Span A\to V\right)$ and $J=\ker\left(V\to\prod_{A\in\Lambda}\Span A\right)$.
In fact we will show that
\[ V=I\oplus J. \]
We have
\[ I=\sum_{A\in\Lambda}\Span A
= \Span\bigcup_{A\in\Lambda} A \]
and
\[ J=\bigcap_{A\in\Lambda}\ker(V\to\Span A)
= \bigcap_{A\in\Lambda}\Span A^c
= \Span\bigcap_{A\in\Lambda} A^c; \]
these are complementary.
\end{proof}

Now we give the proof of:
\begin{namedthm}[Lemma~\ref{grhook}]
\begin{enumerate}[(a)]
\item For $0<i<r(2p-3)$,
\[ H^i(\gr U_n)^T = 0. \]
\item For $0<i<r(2p-3)$, and $\ell<m$,
\[ H^i(\gr K_{\ell m})^T = 0. \]
\item The restriction map
\[ H^{r(2p-3)}(\gr U_n)^T \to
\prod_{\ell<m} H^{r(2p-3)}(\gr K_{\ell m})^T \]
is injective.
\end{enumerate}
\end{namedthm}

\begin{proof}
For this proof, we take cohomology with coefficients extended to $\F_{p^r}$, which does not affect any of the statements.\footnote{Because $T$-invariants commute with extension of the coefficient field}  Then, by Proposition~\ref{vectorspacecoh} (and the remarks of section~\ref{filtration}), $H^*\gr U_n$ has a basis of monomials in the exterior (resp.~polynomial) generators $x_{ijk}$ (resp.~$y_{ijk}$), in degree 1 (resp.~2), which are eigenvectors for the action of $T$.  Here $1\leq i<j\leq n$ and $0\leq k<r$.  The $T$-weight of either $x_{ijk}$ or $y_{ijk}$ is
\[ (t_1,\dots,t_n)\mapsto (t_i/t_j)^{p^k}. \]
Since these monomials form a $T$-eigenbasis, $H^*(\gr U_n)^T$ is spanned by the $T$-invariant monomials.

We claim that there are no such invariant monomials below degree $r(2p-3)$, and that all those in degree $r(2p-3)$ are all pulled back from some hook group $\gr K_{\ell m}$.
We prove the claim by induction on $n$.
To that end, suppose that
\[ \prod_{i<j}z_{ij}=
   \prod_{i<j}\prod_{k}x_{ijk}^{e_{ijk}}y_{ijk}^{a_{ijk}} \]
is a nontrivial $T$-invariant monomial, where $e_{ijk}\in\{0,1\}$ and $a_{ijk}\in\bbN$.
Let
\[ a_{ij}=\sum_k a_{ijk},\quad e_{ij}=\sum_k e_{ijk}, \+ b_{ij}=e_{ij}+a_{ij}. \]
For later use, remember that the degrees of such monomials
\[ z_{ij}=\prod_{k=0}^{r-1}x_{ijk}^{e_{ijk}}y_{ijk}^{a_{ijk}} \] are
\[ \deg(z_{ij})=e_{ij}+2a_{ij}=b_{ij}+a_{ij}, \]
and that each \[ e_{ij}\leq r. \]

Now, we may assume without loss of generality that
\[ \sum_{j>1}b_{1j}, \sum_{i<n}b_{in}>0. \]
(Otherwise, our monomial is pulled back from a copy of $\gr U_{n-1}$, and we may apply the inductive hypothesis).
Then by considering the action of $\F_{p^r}^\times \times1\times\cdots\times1\subset T$, we know that
\[ (p^r-1)\mid\sum_{1<j}\sum_{k=0}^{r-1} p^k(e_{1jk}+a_{1jk}) \neq0, \]
whence it follows (using Lemma~\ref{Quillenlemma}) that 
\[ \sum_{1<j} b_{1j} \geq r(p-1). \]
Similarly,
\[ \sum_{i<n} b_{in} \geq r(p-1). \]
But then the degree of our monomial is
\begin{align*}\numberthis\label{grinequality}
\deg\left(\prod_{i<j}z_{ij}\right)&=\sum_{i<j}\deg(z_{ij})\\
&\geq\sum_{i=1\mbox{ or }j=n}\deg(z_{ij})\\
&=\sum_{i=1\mbox{ or }j=n}(b_{ij}+a_{ij})\\
&\geq\sum_{i=1\mbox{ or }j=n}(b_{ij})+a_{1n}\\
&=\sum_{i<n}(b_{in}) + \sum_{1<j}(b_{1j})-b_{1n}+a_{1n}\\
&=\sum_{i<n}(b_{in}) + \sum_{1<j}(b_{1j})-e_{1n}\\
&\geq r(p-1) + r(p-1) - r = r(2p-3).
\end{align*}
Hence there are no $T$-invariant monomials in degrees $0<i<r(2p-3)$.
Now observe that, in order for equality to occur in the first of the above inequalities, we must have $b_{ij}=0$ whenever $i\neq1$ and $j\neq n$.  That is, our monomial is pulled back from $\gr K_{1n}$.  This proves the claim.

Now, we have established part (a); part (b) follows from (a) because $\gr K_{\ell m}$ is a $T$-equivariant retract of $\gr U_n$.
For part (c), we know that every $T$-invariant monomial in degree $r(2p-3)$ is pulled back from one of the retracts $\gr K_{\ell m}$.
In other words,
\[ \bigoplus_{\ell<m} H^{r(2p-3)}(\gr K_{\ell m})^T\to H^{r(2p-3)}(\gr U_n)^T \]
is surjective.
We claim this is equivalent to the statement of part (c).  This follows from Lemma~\ref{kercoker} because the restriction and pullback maps respect the basis of $T$-invariant monomials in $H^{r(2p-3)}(\gr U_n)^T$.
\end{proof}

Now we'll consider the hook subgroup $K\leq GL_n\F_p$, proving:
\pagebreak[3]
\begin{namedthm}[Lemma~\ref{gressentialhook}]  Let $r=1$.
\begin{enumerate}[(a)]
\item When $0<i<2p-3$,
\[ H^i(\gr L_k)^T=0 \]
for all $1<k<n$;
\item We have
\[ \dim\ker\left( H^{2p-3}(\gr K)^T\to\prod_{1<k<n}H^{2p-3}(\gr L_k)^T \right)=
\begin{cases} 1&\mbox{for }2\leq n\leq p, \\ 0&\mbox{for }n>p. \end{cases} \]
\end{enumerate}
\end{namedthm}

\begin{proof}
We continue all the notation of the previous proof.  Part (a) follows as before from $H^i(\gr U_n)^T=0$, since $\gr L_i$ is a $T$-equivariant retract of $\gr U_n$.

Now, for $K=K_{1n}$, consider $H^{2p-3}(\gr K)^T$.  Suppose as before that
\[ \prod_{i=1\mbox{ or }j=n}z_{ij}=
   \prod_{i=1\mbox{ or }j=n}
     x_{ij}^{e_{ij}}y_{ij}^{a_{ij}} \]
is a nontrivial $T$-invariant monomial.
In Equation~\eqref{grinequality} above, we then have equality in all three inequalities.  Equality in the second one means that
\[ a_{ij}=0\mbox{ except for }(i,j)=(1,n). \]
Equality in the third implies that
\[ e_{1n}=r=1 \]
and
\[ p-1 = \sum_{1<j}(b_{1j}) = \sum_{1<j<n}(e_{1j})+a_{1n}+1. \]
Furthermore, by considering the action of the $i$th factor $1\times\cdots\times \F_p^\times \times\cdots\times1\subset T$, we know that
\[ (p-1)\mid(e_{1i}-e_{in}). \]
Since $e_{ij}\in\{0,1\}$ and $p-1>1$, this implies that
\[ e_{1i}=e_{in} \]
for all $i$.
Putting all this together, we see that our invariant monomial must be of the form
\[ \prod_{i\in S}(x_{1i}x_{in})\cdot x_{1n}y_{1n}^{p-|S|-2} \]
for some $S\subset\{2,3,\dots,n-1\}$.
The point is that, if $S\neq\{2,3,\dots,n-1\}$, then our monomial is pulled back from a retract $\gr L_i$ of $\gr K$ (where $i\notin S$).
When $n\leq p$, there is a unique monomial---corresponding to $S=\{2,...,n-1\}$---which is not pulled back from any $\gr L_i$: namely,
\[ \prod_{1<i<n}(x_{1i}x_{in})\cdot x_{1n}y_{1n}^{p-n}. \]
When $n>p$, there is no such monomial.  Therefore,
\[ \dim\coker\left(\bigoplus_{1<i<n}H^{2p-3}(\gr L_i)^T\to H^{2p-3}(K)^T\right)
= \begin{cases} 1&\mbox{for }2\leq n\leq p, \\ 0&\mbox{for }n>p. \end{cases} \]
This is equivalent to the statement of part (b), using Lemma~\ref{kercoker} as before.
\end{proof}

\end{document}